\documentclass{amsart}
\usepackage{amssymb,amsmath,graphicx,mathtools}
\usepackage[latin1]{inputenc}
\usepackage[normalem]{ulem}
\usepackage{url}
\usepackage{enumerate}

\newtheorem{ThmIntro}{Theorem}
\newtheorem{CorIntro}[ThmIntro]{Corollary}

\newtheorem{thm}{Theorem}[section]
\newtheorem{cor}[thm]{Corollary}
\newtheorem{lemma}[thm]{Lemma}

\newtheorem{prop}[thm]{Proposition}
\theoremstyle{definition}

\newtheorem{question}[thm]{Question}

\theoremstyle{remark}
\newtheorem{rem}[thm]{Remark}

\newtheorem{ex}[thm]{Example}

\numberwithin{equation}{section}

\newcommand{\acts}{\ensuremath{\curvearrowright}}%

\newcommand{\K}{\mathbf{K}}
\newcommand{\bO}{\mathbf{O}}
\newcommand{\bH}{\mathbf{H}}
\newcommand{\N}{\mathbf{N}}

\newcommand{\R}{\mathbf{R}}
\newcommand{\T}{\mathbf{T}}
\newcommand{\cA}{\mathcal{A}}
\newcommand{\cO}{\mathcal{O}}
\newcommand{\cB}{\mathcal{B}}
\newcommand{\C}{\mathbf{C}}

\newcommand{\sign}{\ensuremath{\mathrm{sgn} }}%

\newcommand{\Aut}{\mathord{\text{\rm Aut}}}
\newcommand{\Isom}{\mathord{\text{\rm Isom}}}
 \newcommand{\id}{\textnormal{id}}

\newcommand{\rd}{\: \mathrm{d}}

\begin{document}

\title{Isometric actions on $L_p$-spaces: \\
dependence on the value of $p$}

\date{}
\author[Marrakchi]{Amine Marrakchi} \address{UMPA, CNRS ENS de Lyon\\ Lyon\\FRANCE}\email{amine.marrakchi@ens-lyon.fr}
\author[de la Salle]{Mikael de la Salle} \address{UMPA, CNRS ENS de Lyon\\ Lyon\\FRANCE}\email{mikael.de.la.salle@ens-lyon.fr}

\begin{abstract} Answering a question by Chatterji--Dru\c{t}u--Haglund, we prove that, for every locally compact group $G$, there exists a critical constant $p_G \in [0,\infty]$ such that $G$ admits a continuous affine isometric action on an $L_p$ space ($0<p<\infty$) with unbounded orbits if and only if $p \geq p_G$. A similar result holds for the existence of proper continuous affine isometric actions on $L_p$ spaces. Using a representation of cohomology by harmonic cocycles, we also show that such unbounded orbits cannot occure when the linear part comes from a measure preserving action, or more generally a state-preserving action on a von Neumann algebra and $p>2$. We also prove the stability of this critical constant $p_G$ under $L_p$ measure equivalence, answering a question of Fisher. We use this to show that for every connected semisimple Lie group $G$ and for every lattice $\Gamma < G$, we have $p_\Gamma=p_G$.
\end{abstract}
\maketitle

\section{Introduction}

The study of affine isometric actions of groups on Banach spaces is an important theme in mathematics that is related to many other topics such as group cohomology, fixed point properties and geometric group theory. The case of actions on Hilbert spaces is very well-studied. For example, it is known that a second countable locally compact group $G$ has an affine isometric action on Hilbert spaces without fixed points (resp.\ proper) if and only if $G$ does not have Kazhdan's property (T) (resp. has the Haagerup property). On the other hand, while the groups $G=\mathrm{Sp}(n,1)$ have property (T), Pansu showed in \cite{MR1086210} that they admit affine isometric actions without fixed points on $L_p(G)$ for all $p > 4n+2$ (and actually proper actions by \cite{MR2421319}). This was generalized by Bourdon and Pajot in \cite{MR1979183}. In \cite{MR2221161}, Yu proved that any hyperbolic group $\Gamma$ admits a proper affine isometric action on $\ell^p(\Gamma \times \Gamma)$ for all $p$ large enough, see also \cite{MR3069362}. For more results and references, we refer to \cite{MR2316269} where a systematic study of affine isometric actions of groups on $L_p$-spaces was undertaken. As suggested by the previous results, it is very natural to expect that for a given group $G$, it should be ``easier" to act isometrically on an $L_p$-space when the value of $p$ gets larger. The main result of this paper confirms this intuition. In this statement as in the whole paper, $L_p$ space means $L_p(X,\mu)$ for a standard measure space $(X,\mu)$.

\begin{ThmIntro} \label{thm:main}
Let $G$ be a topological group. Take $0 < p \leq q < \infty$. Then for every continuous affine isometric action $\alpha : G \curvearrowright L_p$, there exists a continuous affine isometric action $\beta : G \curvearrowright L_q$ such that $\| \alpha_g(0) \|^p_{L_p}=\| \beta_g(0) \|^q_{L_q}$ for all $g \in G$.
\end{ThmIntro}

Theorem \ref{thm:main} implies in particular that if a group $G$ has a continuous action by isometries on an $L_p$ space with unbounded (respectively metrically proper) orbits, then it has such an action on an $L_q$ space.

\begin{CorIntro} \label{cor:main}
For every topological group $G$,
\begin{enumerate}[ \rm (i)]

\item The set of values of $p \in (0,\infty)$ such that $G$ admits a continuous action by isometries on an $L_p$ space with unbounded orbits is an interval of the form $(p_G,\infty)$ or $[p_G,\infty)$ for some $p_G \in \{0\} \cup [2,\infty]$.
\item The set of values of $p \in (0,\infty)$ such that $G$ admits a proper continuous action by isometries on an $L_p$ space is an interval of the form $(p'_G,\infty)$ or $[p'_G,\infty)$ for some $p'_G \in \{0\} \cup [2,\infty] $.
\end{enumerate}
\end{CorIntro}

Recall that for $1 \leq p<\infty$, an action by isometries on an $L_p$ space has a fixed point if and only if it has bounded orbits (\cite[Lemma 2.14]{MR2316269} for $p \neq 1$, \cite{MR2929085} for $p=1$). So Corollary \ref{cor:main} answers a question raised in \cite{MR2671183}, sometimes refered to as Dru\c{t}u's conjecture \cite{MR3382026,lavyOlivier}. A partial answer for $\ell_p$ spaces had already been obtained independantly by Czuro\'{n} \cite{MR3590529} and Lavy--Ollivier \cite{lavyOlivier}. In \cite{lavyOlivier} actions coming from ergodic probability measure preserving actions were also covered. Unlike these previous results, it is worth mentionning that in Theorem \ref{thm:main}, the linear part of the action $\beta$ that we construct is very different from the linear part of the original action $\alpha$. We refer to Theorem \ref{main precise form} for a more precise statement.

When $G$ is a second countable locally compact group, it is known that $p_G > 0$ if and only if $G$ has property (T) \cite{MR2316269}, in which case we must have $p_G \geq 2$ (in fact, by an argument of Fisher and Margulis from \cite{MR2198325} that appears in \cite{MR2316269}, one can even show that $p_G > 2$ and that $G$ admits a continuous action by isometries on an $L_p$ space without fixed points for $p=p_G$, see also \cite{lavyOlivier,MR3753580,MR4014781}). Similarly, it is known that $p'_G > 0$ if and only if $G$ does not have the Haagerup property, in which case $p'_G \geq 2$ \cite{MR2270593}. This last fact, as well as the fact that $p_G \notin (0,2)$, if often stated for second countable locally compact groups, but they are true for arbitrary topological groups, see Proposition \ref{gaussian} (but the fact that $p'_G \notin(0,2)$ is meaningful only for locally compact groups, as $p'_G=\infty$ trivially for non locally compact groups). The critical constant $p_G$ and $p'_G$ are different in general. For example, if $G$ is a locally compact group that has neither Kazhdan's property (T) nor Haagerup property, then $p_G=0$ and $p'_G \geq 2$. It is also known that, among Gromov-hyperbolic groups, the value of $p_G$ is unbounded, and explicit lower bounds have been obtained for random groups \cite{MR3872847} (see also \cite{L2spectralGap}).

The linear part the action $\beta$ constructed in Theorem \ref{thm:main} comes from an action on $(X,\mu)$ preserving an infinite measure. It is not possible to achieve the same with a finite measure. Indeed, when $G$ has property (T), any affine action on $L_p$ ($1\leq p<\infty$) whose linear part comes from a probability measure preserving action has a fixed point. This is known when $G$ is discrete \cite{lavyOlivier} or when $G$ admits a finite Kazhdan set \cite{CzuronKalentar}. In general, this is a particular case of the following result dealing with non-commutative $L_p$ spaces. Its proof relies on a general observation of independant interest: under a spectral gap and uniform convexity assumption, any cohomology class with values in an isometric representation has a unique harmonic representent (Lemma \ref{harmonic representent}).
\begin{ThmIntro}\label{thm:NCLp} Let $G$ be a locally compact property (T) group. An action $\alpha : G \curvearrowright L_p(M)$ by affine isometries on the non-commutative $L_p$ space of a von Neumann algebra $M$ has a fixed point in the following two cases:
  \begin{itemize}
  \item $1<p\leq 2$ and $M=B(H)$ for a Hilbert space $H$.
  \item $2 \leq p <\infty$ and the linear part of $\alpha$ comes from an action by automorphisms of $M$ preserving a faithful normal state.
  \end{itemize}
\end{ThmIntro}
We insist that the first conclusion of the theorem is not trivial. It is indeed an open question by Masato Mimura whether a locally compact property (T) group can have an unbounded action by isometries on a non-commutative $L_p$ space for $p<2$ (as explained above this is not possible for usual $L_p$ spaces), and the previous result provides a negative answer for Schatten classes.

It is known that the metric space $(L_p,\|\cdot\|_p^{\frac p q})$ isometrically embeds into $(L_q,\|\cdot\|_q)$ when $p \leq q$ \cite[Remark 5.10]{MR2101227}. This implies the following inequalities for the compression exponents \cite{MR2783928} of a compactly generated group $G$
\begin{equation}\label{eq:compression_exponent} \forall 0<p < q<\infty, \; p \alpha_p(G) \leq q \alpha_q(G).
  \end{equation}
As a consequence of Theorem \ref{thm:main}, we obtain the same inequalities for the equivariant compression exponents.
\begin{equation}\label{eq:equiv_compression_exponent} \forall 0<p < q<\infty, \; p \alpha_p^\#(G) \leq q \alpha_q^\#(G).
  \end{equation}
This inequality is often strict, see \cite{MR2783928}.

We also note that Theorem \ref{thm:main} can be applied to the whole isometry group of $L_p$ and this yields the following corollary.
\begin{CorIntro} \label{cor:isom embedding}
Take $0 <p \leq q < \infty$. Then $\Isom(L_p)$ is isomorphic as a topological group to a closed subgroup of $\Isom(L_q)$.
\end{CorIntro}

Note that if $p > 2$, the subgroup of translations $L_p \subset \Isom(L_p)$ is not unitarily representable \cite[Theorem 3.1]{MR2353912}. In particular, it cannot be embedded as a closed subgroup of $\Isom(L_2)$, which is unitarily representable (by using the affine Gaussian functor \cite[Proposition 4.8]{arano2019ergodic} for example).

With Corollary \ref{cor:main} in hand, we can rephrase a Problem asked by Gromov in \cite[\S $6.D_3$]{MR1253544} as \emph{Given a group $\Gamma$, find the value of $p_G$ and $p'_G$}. The last section of this paper is devoted to this question, in the specific case of Lie groups and their lattices. We first show that for any connected semisimple Lie group $G$, the constants $p_G$ and $p'_G$ depend only on its Lie algebra $\mathfrak{g}$ and not on $G$ itself (see Theorem \ref{depend on lie algebra}). Moreover, we show how this constants can be in principle computed from the decomposition of $\mathfrak{g}$ into simple Lie algebras (see Example \ref{example lie algebras}). When $G$ is simple, we have $p_G=p'_G$ (see Theorem \ref{simple Lie group}). Pansu's result shows that if $G=\mathrm{Sp}(n,1)$ then $2 < p_G \leq 4n+2$ while if $G$ has real rank $\geq 2$, it is known that $p_G=\infty$ \cite[Theorem B]{MR2316269}. Our most technical result is the following one.

\begin{ThmIntro} \label{intro lattice stability}
Let $G$ be a connected semisimple Lie group and let $\Gamma < G$ be a lattice. Then $p_\Gamma=p_G$ and $p'_\Gamma=p'_G$.
\end{ThmIntro}

The proof of this theorems relies on two ingredients. The first one is the following result which shows that the constants $p_G$ and $p'_G$ behave nicely with respect to $L_p$-measure equivalence (see Theorem \ref{thm:relation_withLpME} and the definitions preceeding it). This answers a question by David Fisher.
\begin{ThmIntro}\label{thm:relation_withLpME} If two compactly generated locally compact groups $G_1$ and $G_2$ are $L_p$ measure equivalent, then the critical constants defined in Corollary \ref{cor:main} satisfy
  \[ \min(p_{G_1},p) = \min(p_{G_2},p) \textrm{ and } \min(p'_{G_1},p) = \min(p'_{G_2},p).\]
\end{ThmIntro}

When $\Gamma < G$ is a lattice, then $G$ and $\Gamma$ are measure equivalent. But by definition, they are $L_p$ measure equivalent if and only if $\Gamma$ is $L_p$ integrable. Thus for the proof of Theorem \ref{intro lattice stability}, we need the following second ingredient.

\begin{ThmIntro}
Let $G$ be a connected semisimple Lie group and let $\Gamma < G$ be a lattice. Then $\Gamma$ is $L_p$ integrable for all $p < p_G$.
\end{ThmIntro}

This paper is organized as follows. After some preliminaries in Section \ref{section:preliminaries}, Theorem \ref{thm:main} and its Corollary \ref{cor:isom embedding} are proved in Section \ref{section:p2} for $p=2$ and Section \ref{section:mainthm} in the general case. Section \ref{section:formal} contains a discussion on the validity of Corollary \ref{cor:main} when the linear part of the action of $L_p$ is fixed. Section \ref{section:harmonic} deals with harmonic cocycles and the proof of Theorem \ref{thm:NCLp}. In Section \ref{section:stability}, stability properties of the constants $p_G$ and $p'_G$ are investigated and in particular, Theorem \ref{thm:relation_withLpME} is proved. The last section is dedicated to the study of $p_G$ and $p'_G$ for connected semisimple Lie groups.

\subsection*{Acknowledgments} We are grateful to David Fisher for his question on $L_p$-measure equivalence which led us to Theorem \ref{thm:relation_withLpME}. We also thank Piotr Nowak for his comments which helped us to improve the presentation of this paper. We thank Romain Tessera for providing us with useful references.

\section{Preliminaries}\label{section:preliminaries}
\subsection*{Nonsingular actions}
Let $(X,\mu)$ be a $\sigma$-finite standard measure space (we will always assume that our measure spaces are standard and we omit the $\sigma$-algebra). We denote by $[\mu]$ the measure class of $\mu$. We denote by $\Aut(X,[\mu])$ the group of all \emph{nonsingular} (preserving the measure class $[\mu]$) automorphisms of $(X,\mu)$ up to equality almost everywhere. It is known that $\Aut(X,[\mu])$ is a Polish group for the topology of pointwise convergence on probability measures: a sequence $\theta_n \in \Aut(X,[\mu])$ converges to the identity if and only if $\lim_n \| (\theta_n)_* \nu - \nu \|_1=0$ for every probability measure $\nu \in [\mu]$. We denote by $\Aut(X,\mu)$ the group of all measure preserving automorphisms of $(X,\mu)$ up to equality almost everywhere. It is a closed subgroup of $\Aut(X,[\mu])$. A continuous nonsingular action $\sigma : G \curvearrowright (X,\mu)$ of a topological group $G$ is a continuous homomorphism $\sigma : G \ni g \mapsto \sigma_g \in \Aut(X,[\mu])$.

\subsection*{Cohomology} Let $\pi : G \curvearrowright V$ be a continuous linear representation of a topological group $G$ on a topological vector space $V$. We denote by $Z^1(G,\pi,V)$ the set of all continuous $1$-cocycles, i.e.\ all continuous maps $c : G \rightarrow V$ such that $c(gh)=c(g)+g \cdot c(h)$ for all $g,h \in G$. We denote by $B^1(G,\pi,V) \subset Z^1(G,\pi,V)$ the set of all coboundaries, i.e.\ cocycles $c$ of the form$c(g)=g \cdot v -v$ for some $v \in V$. 

Let $\sigma : G \curvearrowright (X,\mu)$ be a continuous nonsingular action of a topological group $G$. Let $A$ be an abelian topological group (here we use the additive notation for $A$ but this might change sometimes when $A=\T$). We denote by $Z^1_\sigma(G,A)$ the set of all $A$-valued $1$-cocycles of $\sigma$, i.e.\ all continuous functions $c : G \mapsto L_0(X,\mu,A)$ such that $c(gh)=c(g)+\sigma_g(c(h))$. Here for every $f \in L_0(X,\mu,A)$, we use the notation $\sigma_g(f)=f \circ \sigma_g^{-1}$. We denote by $B^1_\sigma(G,A)$ the set of all $1$-coboundaries, i.e.\ cocycles of the form $g \mapsto \sigma_g(f)-f$ for some $f \in L_0(X,\mu,A)$. Finally, we denote by $H^1_\sigma(G,A)=Z^1_\sigma(G,A)/B^1_\sigma(G,A)$ the cohomology group of $\sigma$. 

\subsection*{Skew-product actions} Let $\sigma : G \curvearrowright (X,\mu)$ be a continuous nonsingular action of a topological group $G$. Suppose that $A$ is a locally compact abelian group and let $m$ be the Haar measure of $A$. Then for every $c \in Z^1_\sigma(G,A)$, we can define a new continuous nonsingular action $\sigma \rtimes c$ of $G$ on $(X \times A, \mu \otimes m)$ by the formula 
$$ (\sigma \rtimes c)_g(x,a)=(gx,a+c(g^{-1})(x)).$$
The action $\sigma \rtimes c$ is called the \emph{skew-product action} of $\sigma$ by $c$. Define a function $h : X \times A \rightarrow A$ by $h(x,a)=a$ for all $(x,a) \in X \times A$. Then, by construction, we have $c(g) \otimes 1=(\sigma \rtimes c)_g(h)-h$. Thus the skew-product action $\sigma \rtimes c$ turns the cocycle $c$ into a coboundary. 

\subsection*{The Maharam extension} Let $\sigma : G \curvearrowright (X,\mu)$ be a continuous nonsingular action of a topological group $G$. Then we can define the Radon-Nikodym cocycle $D \in Z^1_\sigma(G,\R^*_+)$ by the formula $ D(g)=\frac{\rd (\sigma_g)_* \mu}{\rd \mu}$ for all $g \in G$. The skew-product action $\widetilde{\sigma}=\sigma \rtimes D : G \curvearrowright X \times \R^*_+$ is called the \emph{Maharam extension} of $\sigma$. Note that $\widetilde{\sigma}$ preserves the measure $\mu \otimes {\rm d} \lambda$ where $\rm d \lambda$ is the restriction to $\R^*_+$ of the Lebesgue measure of $\R$.

\subsection*{Isometric actions on $L_p$-spaces}
Take $p > 0$ and let $(X,\mu)$ be a $\sigma$-finite measure space. For every $\theta : \Aut(X,[\mu])$, we define a linear isometry of $L_p(X,\mu)$ given by
$$ f \mapsto \left( \frac{\theta_* \mu}{\mu} \right)^{1/p} \theta(f).$$
The group $L_0(X,[\mu],\T)$ also acts by multiplication on $L_p(X,\mu)$. We thus obtain a continuous linear isometric representation $$\pi^{p,\mu} : \Aut(X,[\mu]) \ltimes L_0(X,[\mu],\T) \rightarrow \cO(L_p(X,\mu)).$$
It follows from the Banach-Lamperti theorem that this map is surjective when $p \neq 2$. Note that if $\nu \in [\mu]$, the canonical isometry $L_p(X,\mu) \rightarrow L_p(X,\nu)$ given by $f \mapsto \left( \frac{\mu}{\nu} \right)^{1/p}f$ is equivariant with respect to the natural actions $\pi^{p,\mu}$ and $\pi^{p,\nu}$ of $ \Aut(X,[\mu]) \ltimes L_0(X,[\mu],\T)$.

Let $\sigma : G \rightarrow \Aut(X,[\mu])$ be a continuous nonsingular action of a topological group $G$. Then $\sigma^{p,\mu} =\pi^{p,\mu} \circ \sigma$ is a continuous linear isometric representation of $G$ on $L_p(X,\mu)$.  Let $D \in Z^1_\sigma(G,R^*_+)$ be the Radon-Nikodym cocycle. Then for every $p > 0$ and every $g \in G$, the isometry $\sigma^{p,\mu}(g)$ is given by the formula 
$$ \sigma^{p,\mu}(g) : L_p(X,\mu) \ni f \mapsto D(g)^{1/p} \sigma_g(f) \in L_p(X,\mu).$$ Take now some cocycle $\omega \in Z^1_\sigma(G,\T)$. Then the map $g \mapsto \omega(g) \sigma^p_g$ is again a continuous linear isometric representation of $G$ on $L_p(X,\mu)$. Conversely, if $p \neq 2$, it follows from the Banach-Lamperti theorem that every continuous linear isometric representation of $G$ on $L_p(X,\mu)$ is of the form $\pi : g \mapsto \omega(g) \sigma^{p,\mu}_g$ for some continuous nonsingular action $\sigma$ of $G$ and some cocycle $\omega \in Z^1_\sigma(G,\T)$.

Let $\alpha$ be an affine isometric action of $G$ on some $L_p$-space. Since the affine isometry group $\Isom(L_p(X,\mu))$ decomposes as a semi-direct product $$\Isom(L_p(X,\mu))=\cO(L_p(X,\mu)) \ltimes L_p(X,\mu)$$ where $L_p(X,\mu)$ acts by translations, we see that $\alpha$ is of the form $\alpha_g(f)=\pi(g)f+c(g)$ where $\pi$ is an isometric linear representation of $G$ on $L_p(X,\mu)$ and $c \in Z^1(G,\pi,L_p(X,\mu))$ is a cocycle. Observe that even when $\pi=\sigma^{p,\mu}$ for some nonsingular action $\sigma : G \curvearrowright (X,\mu)$, we do \emph{not} have $Z^1(G,\sigma^{p,\mu},L_p(X,\mu)) \subset Z^1_\sigma(G,\C)$ unless $\sigma$ preserves the measure $\mu$.

\section{The case $p=2$}\label{section:p2}
Let $G$ be a topological group and let $p > 0$. We denote by $K^p(G)$ the set of all continuous functions $\psi : G \rightarrow \R_+$ of the form $\psi(g)=\| \alpha_g(0)\|_{L_p}^p$ for some continuous affine isometric action $\alpha$ of $G$ on some $L_p$-space. Note that $K^2(G)$ is the set of all continuous functions on $G$ that are conditionally of negative type.

By using the Gaussian functor, one has the following (classical) result.
\begin{prop} \label{gaussian}
Let $G$ be a topological group. Take $\psi \in K^2(G)$ and $p > 0$. Then there exists a continuous probability measure preserving action $\sigma : G \curvearrowright (X,\mu)$ and a cocycle $c \in Z^1_\sigma(G,\R)$ such that $ \psi(g)^{\frac p 2}=\| c(g) \|_{L_p}^p$ for all $g \in G$. In particular, $\psi^{\frac p 2} \in K^p(G)$ for all $p > 0$.
\end{prop}
\begin{proof}
By definition, there exists an orthogonal representation $\pi : G \rightarrow  \mathcal{O}(H)$ on some Hilbert space $H$ and a cocycle $c \in Z^1(G,\pi, H)$ such that $\psi(g)=\|c(g)\|^2$ for all $g \in G$. Let $\sigma_\pi : G \curvearrowright (X,\mu)$ be the Gaussian action associated to $\sigma$. This means that there exists a linear map $\xi \mapsto \widehat{\xi} \in L_0(X,\mu,\R)$ such that $\widehat{\xi}$ is a centered Gaussian random variable of variance $\| \xi \|^2$ for all $\xi \in H$, and that $\sigma_{\pi}(\widehat{\xi})=\widehat{\pi(g) \xi}$ for all $\xi \in H$. Let $\widehat{c}(g)=\widehat{c(g)} \in L_p(X,\mu)$ for all $g \in G$. Then $\widehat{c}$ is a cocycle for $\sigma_{\pi}$ and a computation shows that $\| \widehat{c}(g) \|^p_{L_p}=C_p \| c(g)\|^p$ for all $g \in G$ and some constant $C_p > 0$.
\end{proof}
\begin{cor}
For every topological group $G$, we have $K^2(G) \subset K^p(G)$ for all $p \geq 2$.
\end{cor}
\begin{proof}
The function $x \mapsto x^\alpha$ is a Bernstein function for all $0 < \alpha \leq 1$. It follows that for every $\psi \in K^2(G)$, we have $\psi^{\alpha} \in K^2(G)$, hence $\psi^{\alpha \frac p 2} \in K^p(G)$. The conclusion follows by taking $\alpha=\frac{2}{p}$.
\end{proof}

\section{Proof of the main theorem}\label{section:mainthm}

\begin{prop} \label{reduction positive}
Let $G$ be a topological group. For every $p > 0$ and every $\psi \in K^p(G)$, there exists a continuous nonsingular action $\sigma : G \curvearrowright (X,\mu)$ and a cocycle $c \in Z^1(G, \sigma^{p,\mu},L_p(X,\mu))$ such that $\psi(g)=\| c(g)\|_{L_p}^p$ for all $g \in G$.
\end{prop}
\begin{proof}
We may assume that $p \neq 2$ thanks to Proposition \ref{gaussian}. By definition, there exists an affine isometric action $\alpha : G \curvearrowright L_p(X,\mu)$ for some probability space $(X,\mu)$ such that $\psi(g)=\| \alpha_g(0)\|_{L_p}^p$ for all $g \in G$. Write $\alpha_g(f)=\pi_g(f)+c(g)$ where $\pi$ is an isometric linear representation of $G$ on $L_p(X,\mu)$ and $c \in Z^1(G,\pi,L_p(X,\mu))$. Since $p \neq 2$, we can write $\pi(g)=\omega(g) \sigma^{p,\mu}_g$ where $\sigma : G \curvearrowright (X,\mu)$ is some nonsingular action and $\omega \in Z^1_\sigma(G,\T)$. Consider the skew-product nonsinglar action $\tilde{\sigma}=\sigma \rtimes \omega : G \curvearrowright (X \times \T, \mu \otimes m)$ where $m$ is the Haar measure of $\T$. Observe that $\tilde{\sigma}_g(u)u^*=\omega(g) \otimes 1$ where $u$ is the function on $X \times \T$ given by $u(x,z)=z$ for all $(x,z) \in X \times \T$. It follows that $\tilde{c}: g \mapsto u c(g)$ defines an element $\tilde{c} \in Z^1(G,\tilde{\sigma}^p, L_p(X \times \T,\mu \otimes m))$ such that $\| \tilde{c}(g)\|_p=\| c(g)\|_p$ for all $g \in G$, where $m$ is the Haar measure of $\T$.
\end{proof}

\begin{lemma} \label{integration lemma}
Take $0 < p < q < \infty$. Let $\varphi : \C \rightarrow \R$ be a nonzero, radial, compactly supported, Lipschitz function. Then there exists a constant $C(q) > 0$ such that for all $w \in \C$, we have
\begin{equation}
\int_{\C} \int_0^\infty |\varphi(z+\lambda^{-1/p}w) - \varphi(z)|^q \rd\lambda \rd z = C(q) |w|^p
\end{equation}
\end{lemma}
\begin{proof}
Let $S$ be the Lebesgue measure of the support of $\varphi$, $M=\| \varphi \|_\infty$ and $K$ the Lipschitz constant of $\varphi$. Then we have $|\varphi(z+\lambda^{-1/p}) - \varphi(z)| \leq \min( 2M,K \lambda^{-1/p})$ for all $z \in \C$ and all $\lambda \in \R^*_+$. Therefore, we have
$$\int_{\C}  |\varphi(z+\lambda^{-1/p}) - \varphi(z)|^q \rd z \leq 2S \min( (2M)^q , K^q \lambda^{-q/p}).$$
Since $q > p$, the function $\lambda \mapsto \min( (2M)^q, K^q \lambda^{-q/p})$ is integrable on $\R^*_+$. Therefore, we can define 
$$C(q)=\int_{\C} \int_0^\infty |\varphi(z+\lambda^{-1/p}) - \varphi(z)|^q \rd\lambda \rd z < +\infty.$$
Since $\varphi$ is not constant, we have $C(q) > 0$. Finally, the statement for all $w \in \C$ follows from the change of variable $\lambda \mapsto |w|^p\lambda$ and the fact that $\varphi$ is radial.
\end{proof}

\begin{thm}\label{main precise form}
Let $G$ be a topological group. Take $0 < p < q < \infty$. Then for every $\psi \in K^p(G)$, there exists a continuous measure preserving action $\sigma : G \curvearrowright (Y,\nu)$ and a function $h \in L_\infty(Y,\nu)$ such that $b(g)=\sigma_g(h)-h \in L_q(Y,\nu)$ with $\psi(g)=\| b(g)\|_{L_q}^q$ for all $g \in G$. In particular $\psi \in K^q(G)$.
\end{thm}

\begin{proof}
By Proposition \ref{reduction positive}, there exists a nonsingular action $\sigma : G \curvearrowright (X,\mu)$ and a cocycle $c \in Z^1(G, \sigma^{p,\mu},L_p(X,\mu))$ such that $\psi(g)=\| c(g)\|_{L_p}^p$ for all $g \in G$. Let $\widetilde{\sigma} :  G \acts (\widetilde{X},\widetilde{\mu})$ be the Maharam extension of $\sigma$. This means that $(\widetilde{X},\widetilde{\mu}) = (X \times \R^*_+,\mu \otimes \rd \lambda)$ where $\rd \lambda$ is the restriction to $\R^*_+$ of the Lebesgue measure of $\R$ and $\widetilde{\sigma} :  G \acts (\widetilde{X},\widetilde{\mu})$ is the measure-preserving action given by $g(x,\lambda) = (gx,\frac{d g^{-1} \mu}{d\mu}(x) \lambda)$. Define $\widetilde c \in Z^1_{\widetilde{\sigma}}(G,\C)$ by the formula $\widetilde c(g,x,\lambda) = \lambda^{-\frac 1 p} c(g,x)$ (observe that $\widetilde{c}$ indeed satisfies the cocycle relation thanks to the term $\lambda^{-1/p}$). Let $\rho : G \acts (Y,\nu)$ be the skew-product action of $\widetilde{\sigma}$ by $\widetilde c$. This is the measure space $(Y,\nu)=(\widetilde X \times \C, \widetilde{\mu} \otimes {\rm d}z ) $ and $\rho$ is the measure preserving action given by $g(\widetilde x,z) =  (g \widetilde x,z+\widetilde c(g^{-1},\widetilde x))$. Let $\varphi\colon \C \to \R$ be a nonzero, radial, compactly supported, Lipschitz function. Define a function $h \in L_\infty(Y,\nu)$ by $h(\widetilde x,z)=\varphi(z)$, and let $b(g) = \rho_g(h) - h$ for all $g \in G$. Then Lemma \ref{integration lemma} shows that $b$ is a cocycle with values in $L_q(Y,\nu)$ that satisfies the conclusion of the theorem up to a constant $C(q) > 0$.
\end{proof}

\begin{rem}
The idea of the proof of Theorem \ref{main precise form} and of the cocycle $\widetilde{c}$ becomes very natural if one uses the notion of modular bundle and Haagerup's canonical $L_p$-spaces as explained in \cite[Sections A.2 and A.3]{arano2019ergodic}. Indeed, by viewing the canonical $L_p$-space $L_p(X)$ as a subspace of $L_0(\mathrm{Mod}(X))$, the isometric linear representation $\sigma^p : G \curvearrowright L_p(X)$ associated to some nonsingular action $\sigma : G \curvearrowright X$ is identified with the restriction to $L_p(X)$ of the Maharam extension $\mathrm{Mod}(\sigma) : G \curvearrowright \mathrm{Mod}(X)$. We can therefore identify every cocycle $c \in Z^1(G,\sigma^p,L_p(X))$ with a cocycle $\widetilde{c} \in Z^1_{\mathrm{Mod}(\sigma)}(G,\C)$. This is crucial in order to be able to use the skew-product construction.
\end{rem}

\begin{proof}[Proof of Corollary \ref{cor:isom embedding}]
Let $G=\Isom(L_p)$ with its canonical affine isometric action on $L_p$. By applying Theorem \ref{thm:main}, we obtain a continuous homomorphism $\Psi : \Isom(L_p) \rightarrow \Isom(L_q)$ such that 
$$ \| g(0) \|_{L_p}^p=\| \Psi(g)(0) \|_{L_q}^q \; \text{ for all } g \in \Isom(L_p).$$
We have to show that $\Psi$ is a homeomorphism on its range. Take $(g_n)_{n \in \N}$ a sequence in $\Isom(L_p)$ and suppose that $\Psi(g_n) \to \id$. We have to show that $g_n \to \id$. Take $f \in L_p$ and let $\tau_f \in \Isom(L_p)$ be the translation by $f$. Then 
$$ \| g_n(f)-f\|^p_{L_p}=\| (\tau_f^{-1} \circ g_n \circ \tau_f)(0)\|^p_{L_p} = \| \Psi(\tau_f^{-1} \circ g_n \circ \tau_f)(0)\|^q_{L_q}.$$
Since $\Psi(g_n) \to \id$, we also have $\Psi(\tau_f^{-1} \circ g_n \circ \tau_f)= \Psi(\tau_f)^{-1} \circ \Psi(g_n) \circ \Psi(\tau_f) \to \id$. This yields
$$ \lim_n \| g_n(f)-f\|^p_{L_p}=\lim_n  \| \Psi(\tau_f^{-1} \circ g_n \circ \tau_f)(0)\|^q_{L_q} =0.$$
Since this holfs for all $f \in L_p$, we conclude that $g_n \to \id$ as we wanted.
\end{proof}

\begin{question} Is Theorem \ref{main precise form} still true when $q=p$? Can we at least realize $\psi$ with an affine isometric action whose linear part comes from a measure preserving action of $G$? One can show that if a measurable function $\varphi : \R \rightarrow \R$ satisfies 
$$ \int_\R \int_0^\infty |\varphi(x+\lambda^{-1/p})-\varphi(x)|^p \rd \lambda \rd x < +\infty$$
for some $p \geq 1$, then $\varphi$ is almost surely constant. Therefore, our method for the proof of Theorem \ref{main precise form} cannot work when $q=p$.
\end{question}

\begin{question} Take $0 < p,q < \infty$. Is it true that $\Isom(L_p)$ embeds as a closed subgroup of $\Isom(L_q)$ if and only if $p \leq \max(2,q)$?
\end{question}

\section{Formal coboundaries}\label{section:formal}
By the Banach-Lamperti theorem, an action by linear isometries on $L_p(X,\mu)$ ($p \neq 2$) comes from a non-singular action on $(X,\mu)$ and a cocyle with values in $\T$, and in particular it corresponds to an action by linear isometries on $L_q(X,\mu)$ for every other $q$. However, in the proof of Theorem \ref{thm:main}, the linear part of the action $\beta$ on $L_q$ is very different from the linear part of the original action $\alpha$. It is natural to wonder when the same linear part can be chosen. The purpose of this section is to investigate this question; the main result shows that this is the case when the translation part is a formal coboundary. This is a variant of \cite{MR3590529,lavyOlivier}.

In this section, we will sometimes change the $\sigma$-algebra. To fix notation we consider a standard measure space $(X,\mu)$ with $\sigma$-algebra $\cA$.
\subsection{Facts on the Mazur maps}
The Mazur map $M_{p,q}$ is the nonlinear map given by $M_{p,q}(f)(x) = \sign(f(x))|f(x)|^{\frac p q}$ for every measurable function $f \colon X \to \C$ (where for $z \in \C \setminus \{0\}$, $\sign(z) = \frac{z}{|z|}$ and $\sign(0)=0$). It is well-known that, for every $1\leq p,q<\infty$, the Mazur map is a uniformly continuous homeomorphism from the unit ball of $L_p$ to the unit ball of $L_q$. This follows from the existence, for every $1 \leq p \leq q<\infty$, of a constant $c = c(p,q)$ such that for all $a,b \in \C$
\begin{equation}\label{eq:mazurLipschitz} |\sign(a) |a|^{\frac p q} - \sign(b) |b|^{\frac p q}| \leq c |a-b|^{\frac p q}\end{equation}
and
\begin{equation}\label{eq:mazurhoelder}  |a-b| \leq c\left(|\sign(a) |a|^{\frac p q} - \sign(b) |b|^{\frac p q}|^{\frac q p}+ |a|^{1-\frac p q}|\sign(a) |a|^{\frac p q} - \sign(b) |b|^{\frac p q}|\right).\end{equation}
For example, we have
\begin{lemma}\label{lem:mazur2} Assume that $(X,\mu)$ is a probability space and let $f \in L_p(X,\mu)$ of mean $0$. Then for every $z \in \C$,
 \[  \|M_{p,q}(f) - z\|_q\geq C(p,q) \|f\|_p^{\frac p q}\]
 for a number $C(p,q) > 0$ depending only on $1 \leq p,q<\infty$.
\end{lemma}
\begin{proof} By homogeneity we can assume that $\|f\|_p=1$. We can assume that $\frac 1 2 \leq |z| \leq \frac 3 2$ as otherwise
 \[\|M_{p,q}(f) - z\|_q \geq \left|\|M_{p,q}(f)\|_q - |z| \right| \geq \frac 1 2.\]
 If $\omega_{q,p}$ is the modulus of uniform continuity of the restriction of the Mazur map $M_{q,p}$ to the ball of radius $\frac 3 2$ in $L_q$, we obtain
 \[ |z|^{\frac q p}\leq \|f-M_{q,p}(z)\|_a\leq \omega_{p,q}(\|M_{p,q}(f)-z\|).\]
 The first inequaliy is just the inequality $|\int F d\mu| \leq \|F\|_p$. So we obtain the lemma with $C(p,q) = \min(\omega_{p,q}^{-1}(2^{- \frac q p}),2^{-1})$.
\end{proof}
If we go outside of the unit ball, the relative position of $p$ and $q$ matter. For example, the next result is not true if $p>q$.
\begin{lemma}\label{lem:mazur} Let $1 \leq p \leq q < \infty$. If $f,g\colon X\to \C$ are two measurable functions such that $f-g \in L_p$, then $M_{p,q}(f) - M_{p,q}(g) \in L_q$.\end{lemma}
\begin{proof}
Replacing $a$ by $f(x)$, $b$ by $g(x)$ in \eqref{eq:mazurLipschitz}, raising to the power $q$ and integrating with respect to $x$, we get
\[ \|M_{p,q}(f) - M_{p,q}(g)\|_{L_q} \leq c \|f-g\|_{L_p}^{\frac p q}.\]
\end{proof}

\begin{lemma}\label{thm:caspmp} Let $1 \leq p \leq q < \infty$. Let $(X,\cA,\mu)$ be a $\sigma$-finite measure space, and $\cB \subset \cA$ be a $\sigma$-subalgebra. Assume that $f \in L_0(X,\cA,\mu)$ satisfies that, for every $h_1 \in L_0(X,\cB,\mu)$, there is $h_2 \in L_0(X,\cB,\mu)$ such that $M_{p,q}(f+h_1)-M_{p,q}(h_2) \in L_q$. Then there is $h \in  L_0(X,\cB,\mu)$ such that $f-h \in L_p$.
\end{lemma}
\begin{proof} Let us first consider the case when $(X,\cB,\mu)$ is purely infinite, that is every element of $\cB$ has measure $0$ or $\infty$. Applying the hypothesis to $h_1=0$, we obtain $h_0$ such that $M_{p,q}(f)-M_{p,q}(h_0) \in L_q$. Applying now the hypothesis to $h_1=-h_0$, we obtain $h_2$ such $M_{p,q}(f-h_0)-M_{p,q}(h_2) \in L_q$. If $h_2\neq 0$, there is $\varepsilon>0$ and $C>0$ such that $\{ |h_2|>\varepsilon\} \cap \{|h_0|\leq C\}$ is not null, that is has infinite measure as it belongs to $\cB$. This implies that $\{ |f-h_0|>\frac \varepsilon 2\} \cap \{|h_0|\leq C\}$ also has infinite measure, as it contains the preceding set minus $\{|M_{p,q}(f-h_0)-M_{p,q}(h_2)|>(2^{\frac p q}-1)\varepsilon^{\frac p q}\}$, which has finite measure because $M_{p,q}(f-h_0)-M_{p,q}(h_2)\in L_q$. But by \eqref{eq:mazurhoelder}, $\{ |f-h_0|>\frac\varepsilon 2\} \cap \{|h_0|\leq C\}$ is contained in $\{|M_{p,q}(f)-M_{p,q}(h_0)|>\delta\}$ whenever $c(\delta^{\frac q p}+C^{1-\frac p q}\delta)<\frac\varepsilon 2\}$. We obtain a contradiction because $M_{p,q}(f)-M_{p,q}(h_0)\in L_q$. So our hypothesis that $h_2\neq 0$ is absurd, so we obtain $M_{p,q}(f-h_0)\in L_q$, or equivalently $f-h_0\in L_p$.

 Let us now consider the case when $(X,\cB,\mu)$ is semifinite. Replacing $\mu$ by $\rho \mu$ for a positive $\rho \in L_1(\cB)$, we can assume that $\mu$ is a probability measure. So we can talk about the $\mu$-preserving conditional expectation $\mathbb{E}_{\cB}$. As in the previous case, there is $h_0\in L_0(\cB)$ such that $M_{p,q}(f)-M_{p,q}(h_0) \in L_q$. In particular $\mathbb{E}_{\cB}(|f|^q|)<\infty$ almost everywhere, and we can define $h=\mathbb{E}_{\cB}(f)$. By the assumption for $h_1=-h$, there is $h_2 \in L_0(B)$ such that $M_{p,q}(f-h)-h_2 \in L_q$. By Lemma \ref{lem:mazur2}, we obtain
 \begin{align*} \|M_{p,q}(f-h)-h_2\|_q^q &= \mathbb{E}(\mathbb{E}_{\cB}(|M_{p,q}(f-h)-h_2|^q))\\ & \geq C(p,q) \mathbb{E}[\mathbb{E}_{\cB}(|f-h|^p))\\& = C(p,q)\|f-h\|_p^p.\end{align*} This proves that $f-h \in L_p$. Finally, the general case follows by decomposing $(X,\cB,\mu)$ as a purely infinite part and a semifinite part.
\end{proof}

\subsection{Formal coboundaries}
Let $(X,\mu)$ be a $\sigma$-finite measure space. Let $G \to \Aut(X,[\mu])\ltimes L_0(X,\mu,\T)$ be a continous group homomorphism. For every $0<p<\infty$, denote by $\pi^{p,\mu} \colon G \to \cO(L_p(X,\mu))$ the corresponding continuous representation. We also denote by $\pi^{p,\mu}$ the continous extension of this linear representation to $L_0(X,\mu)$ (by using the same formula for $\pi^{p,\mu}$ as in the preliminaries). We consider the two cohomology spaces $H^1(G,\pi^{p,\mu},L_p(X,\mu))$ and $H^1(G,\pi^{p,\mu}, L_0(X,\mu))$. Clearly, since $L_p(X,\mu) \subset L_0(X,\mu)$, we have a natural homomorphism 
$$\eta : H^1(G,\pi^{p,\mu},L_p(X,\mu)) \rightarrow H^1(G,\pi^{p,\mu}, L_0(X,\mu)).$$
We are interested in its kernel $H^1_\sharp(G,\pi^{p,\mu},L_p(X,\mu)) := \ker \eta$ which corresponds to cocycles in $L_p(X,\mu)$ that are formal coboundaries. Note that both cohomology groups $H^1(G,\pi^{p,\mu},L_p(X,\mu))$ and $H^1_\sharp(G,\pi^{p,\mu},L_p(X,\mu))$ do not depend on the choice of $\mu$ in the class $[\mu]$ because the map $f \mapsto \left( \frac{\mu}{\nu} \right)^{1/p} f$ intertwines $\pi^{p,\mu}$ and $\pi^{p,\nu}$ for $\nu \in [\mu]$. 

We were unable to determine if the set $\{ p > 0 \mid H^1(G,\pi^{p,\mu},L_p(X,\mu))=0 \}$ is an interval in general. However, we have the following result.

\begin{thm}\label{thm:monotonicity_formal_coboundaries} Let $1 \leq p<q$ and $G \to \Aut(X,[\mu])\ltimes L_0(X,\mu,\T)$ as above. If $H^1_\sharp(G,\pi^{q,\mu},L_q(X,\mu)) =0$ then $H^1_\sharp(G,\pi^{p,\mu},L_p(X,\mu)) =0$.
\end{thm}
\begin{proof} 
For $f \in L_0(X,\mu)$, denote $\partial_{p} f(g) = \pi^{p,\mu}(g) f -f$. By assumption, for all $f \in L_0(X,\mu)$, we have
\begin{equation}\label{eq:formalcoboun_are_cobound q} \forall g \in G, \; \partial_{q} f(g) \in L_q(X,\mu)  \implies \exists h \in L_q(X,\mu), \; \partial_q f=\partial_q h.\end{equation}
We want to prove that for all $f \in L_0(X,\mu)$, we also have
\begin{equation}\label{eq:formalcoboun_are_cobound p} \forall g \in G, \; \partial_{p} f(g) \in L_p(X,\mu)  \implies \exists h \in L_p(X,\mu), \; \partial_p f=\partial_p h.\end{equation}
In order to prove this, we use the properties of the Mazur map. Indeed, it is obvious from the definition of $\pi^{q,\mu}$ that $M_{p,q}$ establishes a continuous (nonlinear) bijection between the spaces $L_0(X,\mu)^{\pi^{p,\mu}(G)}$ and $L_0(X,\mu)^{\pi^{q,\mu}(G)}$, with inverse $M_{q,p}$. Moreover, the map $M_{p,q}$ sends $L_p(X,\mu)$ onto $L_q(X,\mu)$ and it is $G$-equivariant, in the sense that it intertwines $\pi^{p,\mu}$ and $\pi^{q,\mu}$.

Let $f \in L_0(X,\mu)$ such that $\partial_p f(g) \in L_p$ for every $g \in G$. Denote $f_1=M_{p,q}(f)$. By Lemma \ref{lem:mazur}, $\partial_q f_1(g)= M_{p,q}(\pi^{p,\mu}(g) f) - M_{p,q}(f)$ belongs to $L_q$. By our assumption, there exists $h_1 \in L_0(X,\mu)^{\pi^{q,\mu}(G)}$ such that $f_1-h_1 \in L_q$. The set $X_0=\{x \in X, |h_1(x)|=0\}$ is $G$-invariant, and the restriction of $f_1$ to $X_0$ belongs to $L_q$. Equivalently the restriction of $f$ to $X_0$ belongs to $L_p$. So replacing $X$ by $X \setminus X_0$, we can assume that $h_1$ is almost everywhere non-zero. Since $h_1$ is invariant under the representation $\pi^{q,\mu}$, this easily implies that $|h_1|^q \mu$ is an invariant measure and that the cocycle in $L_0(X,\mu,\T)$ is the coboundary of $\mathrm{sgn}(h_1)=h_1/|h_1| \in L_0(X,\mu,\T)$. So without loss of generality we can assume that the homomorphism $G \rightarrow \Aut(X,[\mu])\ltimes L_0(X,\mu,\T)$ takes its values in $\Aut(X,\mu)$. Denote by $\mathcal B$ the $\sigma$-algebra of $G$-invariant sets. Then $L_0(X,\mu)^{\pi^{q,\mu}(G)}$ coincides with $L_0(X,\mathcal B,\mu)$. The theorem is now just a reformulation of Lemma \ref{thm:caspmp}.
\end{proof}

\begin{question}
Is Theorem \ref{thm:monotonicity_formal_coboundaries} still true if we replace $H_\sharp^1(G,\pi^{p,\mu},L_p(X,\mu))$ by $H^1(G,\pi^{p,\mu},L_p(X,\mu))$?
\end{question}

We observe that it is not true in general that 
$$H_\sharp^1(G,\pi^{p,\mu},L_p(X,\mu))=H^1(G,\pi^{p,\mu},L_p(X,\mu)).$$ For example, let $\rho : G \rightarrow \mathcal{O}(H)$ be an orthogonal representation and $c \in Z^1(G,\rho,H)$. Let $\sigma$ be the Gaussian action associated to $\rho$ and $\widehat{c}$ the cocycle for $\sigma$ associated to $c$ as in the proof of Proposition \ref{gaussian}. Then one can show that the cocycle $\widehat{c}$ is \emph{never} a formal coboundary, unless $c$ itself is a coboundary. 

The following proposition also emphasizes the difference between $H_\sharp^1$ and $H^1$.

\begin{prop}
Let $\sigma : G \curvearrowright (X,\mu)$ be a probability measure preserving action. Fix $p \geq 1$ and let $c \in Z^1(G,\pi,L_p(X,\mu))$ where $\pi$ is the Koopman representation of $\sigma$. Suppose that $c(g)=\sigma_g(f) -f$ for all $g \in G$ and some $f \in L_0(X,\mu)$. Then $c \in \overline{B^1(G,\pi,L_p(X,\mu))}$.

In particular, if $\pi$ has spectral gap, then $H_\sharp^1(G,\pi^{p,\mu},L_p(X,\mu))=0$.
\end{prop}
\begin{proof}
Define a sequence of functions $f_n=\max(-n,\min(f,n))$. Let $c_n(g)=\sigma_g(f_n)-f_n$. Observe that $c_n \in B^1(G,\pi,L_p(X,\mu))$ for all $n$ because $f_n$ is bounded. Since $f_n$ converges in measure to $f$ when $n \to \infty$, we have that $c_n(g)$ converges in measure to $c(g)$ when $n \to \infty$. Observe also that $|c_n(g)| \leq |c(g)|$ for all $g \in G$. Thus, by the dominated convergence theorem, we know that $c_n(g) \to c(g)$ in $L_p(X,\mu)$ for all $g \in G$. This proves that the sequence of coboundaries $c_n$ converges to $c$ in $Z^1(G,\pi, L_p(X,\mu))$, as we wanted.
\end{proof}

\section{Harmonic cocycles and state preserving actions}\label{section:harmonic}

Let $E$ be a uniformly convex Banach space and let $\pi$ be a continuous representation of a locally compact group $G$ on $E$. Then we can decompose $E$ as a $\pi$-invariant direct sum $E=E_\pi \oplus E^\pi$ where $E^\pi$ is the subspace of $\pi$-invariant vectors and $E_\pi$ is its natural complement (defined as the orthogonal of $(E^*)^{\pi^*}$ where $\pi^*$ is the dual representation of $\pi$ on $E^*$), see \cite{MR2316269}. We say that $\pi$ has \emph{spectral gap} if $\pi|_{E_\pi}$ has no almost invariant vectors. By \cite[Theorem 1.1]{DruNo}, this is equivalent to the existence of a symmetric compactly supported probability measure $\mu$ on $G$ such that $\|\pi(\mu)|_{E_\pi}\| < 1$.

\begin{lemma} \label{harmonic representent}
Let $G$ be a locally compact group and let $\pi$ be a representation of $G$ on a uniformly convex Banach space $E$ that has spectral gap. Let $\mu$  be a symmetric compactly supported probability measure $\mu$ on $G$ such that $\| \pi(\mu)|_{E_\pi} \| < 1$. Then every cohomology class in $H^1(G,\pi,E)$ admits a unique $\mu$-harmonic representent.
\end{lemma}
\begin{proof}
By decomposing $E=E_\pi \oplus E^\pi$, we can reduce the problem to the case where $\pi$ is either trivial  or has no invariant vectors.

Assume first that $\pi$ is trivial. Then $B^1(G,\pi,E)=0$ and an element of $Z^1(G,\pi,E)$ is just a group homomorphism from $G$ to $E$. Thus, since $\mu$ is symmetric, every element of $Z^1(G,\pi,E)$ is $\mu$-harmonic. 

Now, assume that $\pi$ has no invariant vectors, i.e.\ $E=E_\pi$. Take $c \in Z^1(G,\pi,E)$. Let $\alpha$ be the affine isometric action of $G$ on $E $ associated to $c$. Then the affine map $\alpha(\mu)=\int_{g \in G} \alpha_g \rd \mu(g)$ is $k$-Lipschitz with $k=\|\pi(\mu)\| <1$. Therefore, $\alpha(\mu)$ has a fixed point $\xi  \in E$ for some $\xi \in E$. Let $c'(g)=c(g)+\pi(g)\xi-\xi$. Then we get $ \int_G c'(g) \rd \mu(g) =\alpha(\mu) \xi-\xi=0$. We conclude that $c'$ is a $\mu$-harmonic representent of the cohomology class of $c$. For the uniqueness part, observe that if we have a coboundary $b(g)=\pi(g)\xi-\xi$ that is $\mu$-harmonic, then $\pi(\mu)\xi=\xi$, hence $\xi=0$ because $\| \pi(\mu) \| < 1$.
\end{proof}

\begin{lemma} \label{injective cohomology}
Let $G$ be a compactly generated locally compact group and let $\pi_1,\pi_2$ be two representations of $G$ on two strictly convex Banach spaces $E_1$ and $E_2$. Suppose moreover that $E_1$ is uniformly convex and that $\pi_1$ has spectral gap. Then every injective continuous $G$-equivariant linear map $\psi : E_1 \rightarrow E_2$ induces an injective map $\psi_* : H^1(G,\pi_1,E_1) \rightarrow H^1(G,\pi_2,E_2)$.
\end{lemma}
\begin{proof}
Since $\pi_1$ has spectral gap and $E_1$ is uniformly convex, we can find a symmetric compactly supported probability measure $\mu$ on $G$ such that $\| \pi_1(\mu)_{E_{\pi_1}} \| < 1$. Since $G$ is compactly generated, we can assume that the support of $\mu$ generates $G$.

Take $\omega \in H^1(G,\pi_1,E_1)$. By Lemma \ref{harmonic representent} $\omega$ admits a $\mu$-harmonic representent $c \in Z^1(G,\pi_1,E_1)$. Then $\psi_*\omega \in H^1(G,\pi_1,E_1)$ is represented by the cocycle $c' : g \mapsto \psi(c(g))$. Note that $c'$ is still $\mu$-harmonic. Suppose that $\psi_*\omega=0$. This means that $c'$ is a coboundary, i.e.\ $c'(g)=\pi_2(g) \xi-\xi$ for some $\xi \in E_2$ and all $g \in G$. Since $c'$ is $\mu$-harmonic, we have $\int_G \pi_2(g)\xi \rd \mu(g) = \xi$. Since $E_2$ is stricly convex, this implies that $\pi_2(g)\xi=\xi$ for all $g$ in the support of $\mu$ hence for all $g \in G$. We conclude that $c'=0$, hence $c=0$ because $\psi$ is injective.
\end{proof}
\begin{rem}\label{rem:pmp_actions}
  The previous lemma applies for example when $\sigma : G \curvearrowright (X,\mu)$ is a continuous \emph{probability measure preserving} action of a compactly generated locally compact group, $(\pi_i,E_i) =(\sigma^{p_i,\mu}, L^{p_i}(X,\mu))$ for $\infty>p_1 \geq p_2>1$ and $\psi$ is the inclusion map $L^{p_1} \to L^{p_2}$. The lemma shows that
  \begin{equation*}\label{eq:pmp_actions}\forall 1<p_1<p_2<\infty, H^1(G,\sigma^{p_2,\mu})=0 \implies H^1(G,\sigma^{p_1,\mu})=0.\end{equation*}
  Indeed, the assumption that $H^1(G,\sigma^{p_2,\mu}) =0$ implies that $\sigma^{p_2,\mu}$ has spectral gap, which implies that $\sigma^{p_1,\mu}$ has spectral gap, see \cite{MR2316269}. In particular, if $G$ is a locally compact property (T) group, $H^1(G,\sigma^{p,\mu})=0$ for every $p\in [1,\infty)$. This was already known when $G$ is discrete \cite{lavyOlivier}, or $G$ admits a finite Kazhdan set \cite{CzuronKalentar}.
\end{rem}

  More generally, the lemma applies to the actions on non-commutative $L_p$ spaces associated to state-preserving actions on von Neumann algebras, and also for actions on Schatten $p$-classes. We start with the latter as it is more elementary.

  For a Hilbert space $H$, we denote by $S_p(H)$ the Schatten $p$-class, that is the space of operators on $H$ such that $\|T\|_p:= (\mathrm{Tr}(|T|^p))^{\frac 1 p}<\infty$. We say that a group has $FS_p$ if for every $H$ and every continous isometric representation $\pi : G \curvearrowright S_p(H))$, $H^1(G,\pi,S_p(H))=0$.
\begin{thm} If $G$ is a $\sigma$-compact locally compact group, then the set of $1<p<\infty$ such that $G$ has $FS_p$ is empty if $G$ does not have property (T), and is an interval containting $(1,2]$ otherwise.
\end{thm}
\begin{proof} Assume that $G$ does not have property (T). By Guichardet's theorem, there is a unitary representation $\pi$ on a Hilbert space $H$ and an unbounded cocycle $b \in Z^1(G,\pi)$. If $\xi \in H^*$ is a unit vector, then the formula $g \cdot T = \pi(g) T + b(g) \otimes \xi$ defines an unbounded action by isometries on $S_p(H)$ for every $1 \leq p \leq \infty$. So $G$ does not have $FS_p$.

  If $G$ has property (T), then $G$ is compactly generated, and has $FS_2$ by Delorme's theorem ($S_2(H)$ is a Hilbert space). We have to prove that, if $1<p<q<\infty$ are such that $G$ has $FS_q$, then $G$ has $FS_p$. We can assume that $p \neq 2$. Let $\pi \colon G \to O(S_p(H))$ be an orthogonal representation. By \cite{MR611284}, $\pi_g$ is of the form $\pi_g(T) = W_g J_g(T)$ for a unitary $W_g$ and a Jordan automorphism $J_g$ of $B(H)$. In particular, the same fomula gives rise to an isometric representation on $S_q(H)$ and the inclusion $S_p(H) \subset S_q(G)$ is equivariant. Moreover both representations have spectral gap, \cite{MR611284}, and the Schatten spaces are uniformly convex when $1<p<\infty$. We conclude by Lemma \ref{injective cohomology} that $H^1(G,S_p(H)) \to H^1(G,S_q(H))=0$ is injective and $H^1(G,S_p(H))=0$.
    \end{proof}

To state our result about state-preserving actions, we need to recall Haagerup's general definition of non-commutative $L_p$ spaces \cite{MR560633}. We follow the approach in \cite[\S 2.1]{MR3987868}. If $M$ is a von Neumann algebra, the \emph{core of $M$} is the unique (up to unique isomorphism) tuple $(c(M),\tau,\theta,\iota)$ of a von Neumann algebra $c(M)$, a normal faithful semifinite trace $\tau$ on $c(M)$, a continuous homomorphism $\theta:\R^*_+ \to \mathrm{Aut}(M)$, and a normal injective $*$-homomorphism $\iota: M \to C(M)$ satisfying $\tau \circ \theta_s = e^{-s}\tau$ and $\{x \in c(M) \mid \forall s \in \R, \theta_s (x)=x\} = \iota(M)$. In the sequel, we will identify $M$ with its image in $c(M)$ and write $x$ instead of $\iota(x)$. For example, if $\phi$ is a normal faithful weight on $M$, the core can be realized as $c(M) = M \rtimes_{\sigma^\phi} \R$, the crossed product by the modular automorphism group of $\phi$. In that case, $\theta$ is given by the dual action and $\iota$ the natural inclusion. For details on the construction and the existence of a trace $\tau$ as required, see \cite{MR1943006}. 
Then $L_p(M)$ is defined as the space of $\tau$-measurable operators affiliated to $c(M)$  such that $\theta_s(h)= e^{-s/p} h$ for all $s \in \R$. By \cite[Theorem 1.2 and 1.3]{MR560633}, for every normal state $\varphi$ on $M$, there is a unique element of $L_1(M)$, that we also denote $\varphi$, satisfying $$\varphi \left(\int_{\R^*_+} \theta_t(x) \frac{dt}{t} \right) = \tau(\varphi x)$$ for every nonnegative $x \in c(M)$. Moreover, this map extends by linearity to an isomorphism $M_* \to L_1(M)$. This allows to define, for $1\leq p<\infty$ and $x \in L_p(M)$, $\|h\|_p := \| |h|^p \|_{M_*}^{\frac 1 p}$. This turns $L_p(M)$ into a Banach space that is uniformly convex if $1<p<\infty$, and the $\|\cdot\|_p$ norms satisfy Hölder's inequality. 

By the uniqueness of the core, any continuous action by automorphisms of $G$ on $M$ gives rise to a continuous action by isometries on $L_p(M)$, that we call the $p$-Koopman representation.

\begin{thm}
Let $G$ be a locally compact group with property (T). Let $\sigma : G \rightarrow \Aut(M)$ be an action on a von Neumann algebra $M$ that preserves some faithful normal state $\varphi$ on $M$.  Take $p \geq 2$ and let $\pi^p : G \rightarrow O(L_p(M))$ be the $p$-Koopman representation associated to $\sigma$. Then $H^1(G,\pi^p,L_p(M))=0$.
\end{thm}
\begin{proof}
  First, $\pi_p$ has spectral gap by \cite{MR2957217}.

Define a continuous linear map $\psi : L_p(M) \rightarrow L_2(M)$ by the formula 
$$\psi(x\varphi^{1/p}) = x\varphi^{1/2}$$
for all $x \in M$. This map is well-defined because $M\varphi^{1/p}$ is dense in $L_p(M)$ and for all $x \in M$, we have $$\| x \varphi^{1/2}\|_2=\| x \varphi^{1/p} \cdot \varphi^{1/q}\|_2 \leq \| x \varphi^{1/p}\|_p \cdot \| \varphi^{1/q}\|_q= \| x \varphi^{1/p}\|_p  $$
where $\frac{1}{p}+\frac{1}{q}=\frac{1}{2}$. Observe that $\pi^p(g)(x\varphi^{1/p})=\sigma(g)(x)\varphi^{1/p}$ for all $x \in M$, because $\sigma$ preserves $\varphi$, and we have the same formula for $p=2$. This implies that $\psi$ is $G$-equivariant with respect to the Koopman representations of $\sigma$. Thus it induces an injective map from $H^1(G,\pi^p,L_p(M))$ into $H^1(G,\pi^2,L_2(M))$ by Lemma \ref{injective cohomology}. Since $G$ has property (T), we know that $H^1(G,\pi^2,L_2(M))=0$. We conclude that $H^1(G,\pi^p,L_p(M))=0$.
\end{proof}

\section{Stability properties of the constants $p_G$ and $p'_G$} \label{section:stability}

We start with some elementary stability properties.
\begin{prop}\label{prop:subgroups_and_quotients}
Let $G$ be a locally compact group and let $H<G$ be a closed subgroup. Then $p'_H \leq p'_G$. If $H \triangleleft G$ is a closed normal subgroup, then $p_{G/H} \geq p_G$.
\end{prop}
\begin{proof} Any proper action of $G$ on an $L_p$ space restricts to a proper action of $H$, so $p'_H \leq p'_G$. If $H$ is normal, any action of $G/H$ with unbounded orbits on an $L_p$ space can be seen as an action of $G$, so $p_{G/H} \geq p_G$.
\end{proof}

\begin{prop}\label{prop:stability_product} Let $G_1,G_2$ be two locally compact groups. Then
  \[ p_{G_1 \times G_2} =\min(p_{G_1},p_{G_2}) \quad  \text{and} \quad p'_{G_1 \times G_2} =\max(p'_{G_1},p'_{G_2}). \]
\end{prop} 
\begin{proof} The inequalities $p_{G_1 \times G_2} \leq \min(p_{G_1},p_{G_2})$ and $p'_{G_1 \times G_2} \geq \max(p'_{G_1},p'_{G_2})$ follow from Proposition \ref{prop:subgroups_and_quotients}.

  For the inequality $p_{G_1 \times G_2} \geq \min(p_{G_1},p_{G_2})$, observe that if $G_1 \times G_2 \curvearrowright L_p$ has unbounded orbits, then its restriction to either $G_1$ or $G_2$ also has unbounded orbits.

  For the inequality $p'_{G_1 \times G_2} \leq \max(p'_{G_1},p'_{G_2})$, observe that given two isometric actions $G_1 \curvearrowright L_p(\Omega_1)$ and $G_1 \curvearrowright L_p(\Omega_2)$, one can construct the action of $G_1 \times G_2$ on $L_p(\Omega_1 \cup \Omega_2)$, which is proper whenever both actions were proper.
\end{proof}
\begin{prop}\label{prop:stability_quotient_by_compact}
 Let $G$ be a locally compact group and let $K \triangleleft G$ be a normal compact subgroup. Then
\[ p_{G/K} = p_G \quad  \text{and} \quad p'_{G/K} = p'_G.\]
\end{prop}
\begin{proof} The only thing to notice is that the space of $K$-invariant vectors for an isometric representation of $G$ on an $L_p$-space is isometric to an $L_p$-space. This either follows from the general form of isometries of $L_p$-space, or from the classical result \cite{MR248514} that the range of a norm $1$ projection in $B(L_p)$ is isometric to an $L_p$-space.
  \end{proof}

We now investigate the stability of the constants $p_G$ and $p'_G$ under measure equivalence. For this we need to recall some definitions.

Let $G$ be a locally compact group $G$ with left Haar measure $m_G$. A measure preserving action of $G$ on $(X,\mu)$ is called \emph{principal} if there is a measure preserving conjugacy between $G \curvearrowright (X,\mu)$ and an action of the form $G \curvearrowright (G \times \Omega, m_G \otimes \nu)$ where $G$ acts by translation of the left coordinate and $(\Omega, \nu)$ is a measure space with finite measure. In that case, we can identify $\Omega$ with $X/G$ and thus equip $X/G$ with the finite measure space structure coming from this identification. The finite measure on $X/G$ coming from this identification will be denoted $\mu/G$.

Let $G_1$ and $G_2$ be two locally compact groups. The groups $G_1,G_2$ are said to be measure equivalent if there is a \emph{measure equivalence coupling}, that is a measure space $(X,\mu)$ equipped with two commuting measure preserving principal actions of $G_1$ and $G_2$. Observe that we then obtain a measure preserving action of $G_2$ on $(X/G_1,\mu/{G_1})$ and of $G_1$ on $(X/G_2,\mu/{G_2})$. We will use the notations $\mu_1=\mu/G_1$ and $\mu_2=\mu/G_2$.

Let $p_1 : X  \rightarrow G_1$ be a $G_1$-equivariant map (it always exists because the action of $G_1$ is principal). Then for every $g_2 \in G$, the map 
$$X \ni \omega \mapsto p_1(g_2 \omega)^{-1}p_1(\omega) \in G_1$$ is $G_1$-invariant. Thus, one can define a cocycle $c_1 : G_2 \times X/G_1 \rightarrow G_1$ by the formula
$$ c_1(g_2,\omega_1)=p_1(g_2 \omega)^{-1}p_1(\omega)$$ where $\omega \in X$ is any representant of $\omega_1 \in X/G_1$. Observe that the cocycle $c_1$ depends on the section $p_1$.

Now, suppose that $G_1$ is compactly generated. Take $p > 0$. We say that the measure equivalence coupling $(X,\mu)$ is \emph{$L_p$-integrable} over $G_1$ if there exists a $G_1$-equivariant map $p_1 : X  \rightarrow G_1$ such that the associated cocycle $c_1$ defined above satisfies the  following $L_p$-integrability condition
$$ \int_{X/G_1} |c_1(g_2,\omega_1)|_{G_1}^p \rd \mu_1(\omega_1)<\infty\ \  \forall g_2 \in G_2$$
where $| \cdot |_{G_1}$ is the word length with respect to any compact generating set of $G_1$ (the integrability condition does not depend on the choice).

We say that two compactly generated groups $G_1$ and $G_2$ are $L_p$-measure equivalent if there exists a measure equivalence coupling that is $L_p$-integrable over both $G_1$ and $G_2$. Thanks to the construction in \cite[Theorem 3.3]{MR1740985}, we know that $L_p$-measure equivalence is more general than \emph{$L_p$-orbit equivalence}.

More details on measure equivalence of locally compact groups and $L_p$-measure equivalence can be found in \cite[Section 1.2]{MR3117525} and the references therein.

%

\begin{thm}\label{invariance_under_LpME} Let $G_1$ and $G_2$ be two locally compact groups. Let $1 \leq  p < \infty$. Suppose that $G_1$ is compactly generated and that there exists a measure equivalence coupling between $G_1$ and $G_2$ that is $L_p$-integrable over $G_1$.
  \begin{enumerate}[\rm (i)]
  \item If $G_1$ admits an affine isometric action without fixed points on some $L_p$-space then so does $G_2$.
  \item If $G_1$ admits a proper affine isometric action on some $L_p$-space then so does $G_2$.
    \end{enumerate}
\end{thm}

\begin{proof}We use the standard tool of induction of Banach-space actions as in \cite[Section 8.b]{MR2316269}. Let $(X,\mu)$ be measure equivalence coupling that is $L_p$-integrable over $G_1$. Let $p_1$ and $c_1$ be as in the definition. 

Let $\alpha \colon G_1 \curvearrowright E$ be an action by affine isometries of $G_1$ on a Banach space $E$. Let $L_0(X,\mu,E)^{G_1}$ be the set of all $G_1$-equivariant Bochner-measurable maps from $X$ to $E$. Note that $L_0(X,\mu,E)^{G_1}$ is only an affine subspace of $L_0(X,\mu,E)$ (it does not contain $0$). Observe that we have a natural affine action $\beta : G_2 \curvearrowright L_0(X,\mu, E)^{G_1}$ given by $\beta_{g_2}(f)(x)=f(g_2^{-1}x)$ for all $x \in X$. If $f,h \in L_0(X,\mu,E)^{G_1}$, then the map $\| f-h\|_E : X \rightarrow \R_+$ is $G_1$-invariant and thus we can define
$$ \| f-h\|_{p,G_1} = \left( \int_{X/G_1} \|f-h\|_E^p \rd\mu_1 \right)^{1/p}.$$
We clearly have $\| \beta_{g_2}(f)-\beta_{g_2}(h) \| = \| f-h\|$ for all $g_2 \in G_2$ (because the action of $G_2$ on $X/G_1$ preserves the measure).

Take $x \in E$ and define $f_0 \in L_0(X,\mu,E)^{G_1}$ by the formula $f_0(\omega)=p_1(\omega) \cdot x$. Now, define an affine subspace $$F=\{ f \in L_0(X,\mu,E)^{G_1} \mid \| f-f_0\|_{p,G_1} < \infty\}.$$
Note that the space $F$ is isometric to $L_p( X/G_1,\mu_1, E)$. In particular, if $E$ is an $L_p$-space than $F$ is also an $L_p$-space.

Now, a key observation is that $\beta_{g_2}(f_0) \in F$ for all $g_2 \in G_2$. Indeed, we have
$$ \| \beta_{g_2}(f_0) - f_0 \|_{p,G_1}=\left( \int_{X/G_1} \| c_1(g_2,\omega_1) \cdot x -x\|_E^p \rd \mu_1(\omega_1) \right)^{1/p}$$
and this integral is finite because of the $L_p$-integrability of $c_1$ and the fact that the function $g_1 \mapsto \| g_1 \cdot x-x\|_E$ grows sublinearly: there exists a constant $C > 0$ such that $\| g_1 \cdot x- x\|_E \leq C|g_1|_{G_1}$ for all $g_1 \in G_1$.

Since $\beta_{g_2}(f_0) \in F$ for all $g_2 \in G_2$, it follows that $F$ is globally invariant under the action $\beta$. We conclude that the action $\beta : G_2 \curvearrowright L_0(X,\mu,E)^{G_1}$ restricts to an affine isometric action, still denoted $\beta$, of $G_2$ on $F$. We call the affine isometric action $\beta : G_2 \curvearrowright F$ the \emph{induced action} of $\alpha$ (note that it depends on the choice of $p_1$).

To prove item $(\rm i)$, it is enough to check that if $\beta$ has a fixed point then $\alpha$ also has a fixed point. Let $f \in F$ be a fixed point for $\beta$. This means that $f : X \rightarrow E$ is a $G_1$-equivariant measurable map that is also $G_2$-invariant. Thus we can view $f$ as a $G_1$-equivariant map from $X/G_2$ to $E$. Since $X/G_2$ admits a $G_1$-invariant  probability measure $\mu_2$, we can push it forward by $f$ to obtain a $G_1$-invariant probability measure on $E$. We conclude by \cite[Lemma 2.14]{MR2316269} that $\alpha$ has a fixed point.

To prove item $(\rm ii)$, it is enough to check that if $\alpha$ is proper then $\beta$ is also proper.
We have
  \[ \| \beta_{g_2}(f_0) - f_0 \|_{p,G_1} = \left(\int_{X/G_1} \| c_1(g_2,\omega_1) \cdot x-x \|_E^p \rd\mu_1(\omega_1)\right)^{\frac 1 p}.\]
Take $R > 0$. Since $\alpha$ is proper, we can choose a compact subset $K_1 \subset G_1$ such that $\| g_1 x-x\| \geq R$ for all $g_1 \in G_1 \setminus K_1$. Observe that
\begin{align*} &  \mu_1 \{ \omega_1 \in X/G_1 \mid c(g_2, \omega_1) \in K_1 \}  = \\
&\frac{1}{m_{G_1}(K_1)}  \mu \{ \omega \in X \mid p_1(\omega) \in K_1 \text{ and } p_1(g_2 \omega)^{-1}p_1(\omega) \in K_1 \}  \leq  \\ 
&\frac{1}{m_{G_1}(K_1)} \mu \{ \omega \in X \mid p_1(\omega) \in K_1 \text{ and } p_1(g_2 \omega) \in K_1^{-1} K_1 \} \end{align*}
Since $\mu(p_1^{-1}(K_1 \cup K_1^{-1} K_1))<+\infty$, there exists a compact subset $K_2 \subset G_2$ such that $\mu(p_1^{-1}(K_1 \cup K_1^{-1}K_1) \setminus p_2^{-1}(K_2) ) \leq \frac{1}{4} m_{G_1}(K_1)$. Take $g_2 \in G_2 \setminus K_2 K_2^{-1}$. Then, we cannot have $p_2(\omega) \in K_2$ $p_2(g_2 \omega) \in K_2$ at the same time. Therefore, we get
$$\mu \{ \omega \in X \mid p_1(\omega) \in K_1 \text{ and } p_1(g_2 \omega) \in K_1^{-1} K_1 \} \leq 2 \mu(p_1^{-1}(K_1 \cup K_1^{-1}K_1) \setminus p_2^{-1}(K_2) )$$
which yields
$$ \mu_1 \{ \omega_1 \in X/G_1 \mid c(g_2, \omega_1) \in K_1 \}  \leq \frac{1}{2}.$$
By the choice of $K_1$, this means that 
$$ \mu_1 \{ \omega_1 \in X/G_1 \mid \| c(g_2, \omega_1) \cdot x- x\|_E \geq R\}  \geq \frac{1}{2}$$
hence
$$ \| \beta_{g_2}(f_0)-f_0 \|_{p,G_1} \geq 2^{-1/p} R.$$
This holds for all $g_2 \in G_2$ outside of the compact subset $K_2 K_2^{-1}$. We conclude that $\beta$ is proper.
\end{proof}

When $G$ is a locally compact group and $\Gamma < G$ is a lattice, then $(G,m_G)$ is measure equivalence coupling between $G$ and $\Gamma$. Therefore, Theorem \ref{invariance_under_LpME} covers in particular the following corollary which is already implicit in \cite{MR2316269}.

\begin{cor}[\cite{MR2316269}]\label{cor:critical_exponents_and_integrable_lattices}
Let $G$ be a locally compact compactly generated group and let $\Gamma < G$ be a lattice. Then $p_{\Gamma} \leq p_G$ and $p'_{\Gamma} \leq p'_G$. If $\Gamma$ is $L_p$-integrable for some $p \geq p_\Gamma$ (resp.\ $p \geq p'_\Gamma$) then $p_G=p_{\Gamma}$ (resp.\ $p'_G=p'_{\Gamma}$).
\end{cor}

\section{Semisimple Lie groups and their lattices}
The following result is a combination of \cite[Corollary
  1.4]{MR2421319} and \cite{MR550072}.
\begin{thm} \label{simple Lie group}
Let $G$ be a non-compact connected simple Lie group with finite center. Let $\alpha : G \curvearrowright E$ be an affine isometric action on some $L_p$-space $E$ with $p > 1$. Then $\alpha$ is either proper or has a fixed point. In particular,  $p_G=p'_G$.
\end{thm}
\begin{proof}
Write $\alpha_g(x)=\pi(g)x + c(g)$ where $\pi$ is the linear part of $E$ and $c \in Z^1(G,\pi,E)$. Decompose $E=E_\pi \oplus E_0$ where $E_\pi$ is the subspace of invariant vectors of $\pi$ and $E_0=\ker P$ where $P : E \rightarrow E_\pi$ is the unique $\pi$-invariant projection. The map $g \mapsto P c(g)$ is a continuous homomorphism from $G$ into an abelian group. Since $G$ is a connected simple Lie group, it must vanish. Thus $c(g) \in E_0$ for all $g \in G$ and we can restrict $\alpha$ to an affine isometric action on $E_0$. By the Banach space version of the Howe-Moore property (due to Veech \cite{MR550072}, see also \cite[Appendix 9]{MR2316269}), we know that $\pi|_{E_0}$ is a $C_0$-representation. We conclude by \cite[Corollary 1.4]{MR2421319}, that $\alpha$ is either proper or has a fixed point. 
\end{proof}

Let $\mathfrak{g}$ be a real semisimple Lie Algebra. Let $\mathrm{Ad}(\mathfrak g)$ denote the adjoint group of $\mathfrak g$, that is the connected component of the identity in $\mathrm{Aut}(\mathfrak g)$, and write $p_{\mathfrak g} = p_{\mathrm{Ad}(\mathfrak g)}$ and $p'_{\mathfrak g} = p'_{\mathrm{Ad}(\mathfrak g)}$. By \cite[Section I.14]{MR1920389}, $\mathrm{Ad}(\mathfrak g)$ is isomorphic to $G/Z(G)$ for every connected Lie group with Lie algebra isomorphic to $\mathfrak{g}$.

\begin{ex} \label{example lie algebras} We can compute $p_{\mathfrak g}$ for most semisimple Lie algebras:
  \begin{itemize}
  \item If $\mathfrak{g}$ is a semisimple Lie algebra, then it is a direct sum of simple Lie algebras $\mathfrak{g} = \oplus_i \mathfrak{g_i}$ and $\mathrm{Ad}(\mathfrak{g}) = \prod_i \mathrm{Ad}(\mathfrak{g})$. Therefore by Proposition \ref{prop:stability_product},
    \[ p_{\mathfrak{g}} = \min_i p_{\mathfrak{g}_i} \quad \text{and} \quad p'_{\mathfrak g} = \max_i p'_{\mathfrak{g}_i}.\]
  \item If $\mathfrak g$ is a simple Lie algebra with real rank $0$, then $p_{\mathfrak g} = \infty$ and $p'_{\mathfrak g}=0$ because $\mathrm{Ad}(\mathfrak{g})$ is compact, so all its continuous actions are both proper and bounded.
  \item If $\mathfrak g$ is a simple Lie algebra with real rank $\geq 2$, then $p_{\mathfrak{g}} = p'_{\mathfrak{g}} = \infty$ by \cite[Theorem B]{MR2316269}.
  \item If $\mathfrak g$ is a simple Lie algebra with real rank $1$, then $p_{\mathfrak{g}} = p'_{\mathfrak{g}} < +\infty$ and we have
  \begin{equation}\label{eq:values_criticalp_rankone} p_{\mathfrak{g}}  = p'_{\mathfrak{g}} \begin{cases} =0 & \textrm{if }\mathfrak{g} \simeq \mathfrak{so}(n,1), \: n \geq 2 \\
  =0 & \textrm{if }\mathfrak{g} \simeq \mathfrak{su}(n,1), \: n \geq 2\\ 
   \in(2, 4n+2] & \textrm{if }\mathfrak{g} \simeq \mathfrak{sp}(n,1), \: n \geq 2\\
      \in (2,22] & \textrm{if }\mathfrak{g} \simeq \mathfrak{f}_4^{-20}.\end{cases}\end{equation}
      Indeed, the equality $p_{\mathfrak{g}} = p'_{\mathfrak{g}}$ follows from Theorem \ref{simple Lie group}. The equality $p_{\mathfrak{g}} =0$ for $\mathfrak{g} \simeq \mathfrak{so}(n,1), \: \mathfrak{su}(n,1)$ comes from the fact that the groups $\mathrm{SO}(n,1)$ and $\mathrm{SU}(n,1)$ do not have property (T). The remaining cases have property (T), so $p_{\mathfrak{g}}>2$, and it follows from Pansu's result \cite{MR1086210} (see Theorem \ref{pansu rank one}) that $p_{\mathfrak{sp}(n,1)} \leq 4n+2$ and $p_{\mathfrak{f}_4^{-20}} \leq 22$. That the algebras appearing in \eqref{eq:values_criticalp_rankone} are all the rank one real simple Lie algebras is Cartan's classification \cite{MR1509178}, see also \cite[Theorem 6.105]{MR1920389}. Observe that the list in \cite{MR1920389} is apparently slightly different as it includes $\mathfrak{su}(1,1)$,  $\mathfrak{sl}_2(\C)$, $\mathfrak{sl}_2(\bH)$ but excludes $\mathfrak{so}(n,1)$ for $n = 2, 3, 5$. But this is the same list thanks to the expectional isomorphisms $\mathfrak{su}(1,1) \simeq \mathfrak{so}(2,1)$, $\mathfrak{sl}(2,\C) \simeq \mathfrak{so}(3,1)$ and $\mathfrak{sl}(2,\bH) \simeq \mathfrak{so}(5,1)$ \cite[Formula (6.110)]{MR1920389}.

  \end{itemize}
\end{ex}

It would be very interesting to compute the exact value of $p_{G}$ for $G=\mathrm{Sp}(n,1)$ or $F_4^{-20}$. From the preceding and Theorem \ref{depend on lie algebra}, this would allow to compute $p_G$ for every connected semisimple Lie group and any lattice in it. It is not even known whether $\lim_n p_{\mathrm{Sp}}(n,1)=\infty$.

      \begin{thm} \label{depend on lie algebra}
      Let $G$ be a connected semisimple Lie group with Lie algebra $\mathfrak{g}$. Then
    \begin{equation}\label{eq:value-group} p_G = p_{\mathfrak{g}} \quad  \text{and} \quad p'_G = p'_{\mathfrak{g}}.\end{equation}
    \end{thm} 
    \begin{proof}
    The Theorem is clear if $G$ is compact, so we can and will assume that $G$ is not compact. We can and will even assume (replacing $G$ by its quotient by its compact factors) that $G$ does not have compact factors.

    Consider first the case when $G$ is a connected simple Lie group. If $G$ has finite center $Z(G)$, then $G/Z(G)$ is isomorphic to $\mathrm{Ad}(\mathfrak g)$ so \eqref{eq:value-group} follows from Proposition \ref{prop:stability_quotient_by_compact}. If $Z(G)$ is infinite, it follows from \cite{MR510552} that either $\mathfrak{g}$ is isomorphic to $\mathfrak{su}(n,1)$ for some $n \geq 1$ and $G$ is the universal covering group of $\mathrm{SU}(n,1)$, or $\mathfrak{g}$ has rank $\geq 2$. In the first case, $G$ has the Haagerup property \cite[Theorem 4.0.1]{MR1852148} so $p_G = p'_G=0$. If $\mathfrak{g}$ has rang $\geq 2$, it is also known that $p_G = p'_G=\infty$, see \cite[Corollary 1.2]{strongTsp4}. We obtain \eqref{eq:value-group} in both cases.

    Consider now the general case when $G$ is semisimple. Denote by $\widetilde G$ the universal cover of $G$, and $\widetilde q \colon \widetilde G \to G$ and $q \colon G \to G/Z(G) \simeq \mathrm{Ad}(G)$ the quotient maps. Let $\mathfrak{g} = \oplus_{i=1}^n \mathfrak{g_i}$ be the decomposition of the Lie algebra of $G$ as a direct sum of simple Lie algebras, and $\widetilde G_i$ the simply connected Lie group with Lie algebra $\mathfrak{g}_i$. We can identify $\widetilde G$ with $\prod_{i=1}^n \widetilde G_i$, and $G/Z(G)$ with $\prod_i \mathrm{Ad}(\mathfrak{g}_i)$. We deduce from the case of simple Lie groups already covered and from Proposition \ref{prop:stability_product} that $p_{\widetilde G} = p_{\mathfrak{g}} = p_{G/Z(G)}$.  On the other hand, Proposition \ref{prop:subgroups_and_quotients} gives $p_{\widetilde G} \geq p_G \geq p_{G/Z(G)}$, so we deduce $p_G = p_{\mathfrak{g}}$. Moreover, the image $G_i:=\widetilde{q}(\widetilde G_i)$ is a closed subgroup of $G$ with Lie algebra $\mathfrak{g_i}$, so we have $p'_{\mathfrak{g}} = \max_i p'_{\mathfrak{g}_i} = \max_i p'_{G_i} \leq p'_G$. So it remains to show that $p'_G \leq p'_{\mathfrak g}$. This is obvious if $p'_{\mathfrak g}=\infty$. When $\mathfrak{p'_{\mathfrak g}}=0$, the $\mathfrak{g_i}$ are isomorphic to $\mathfrak{su}(n,1)$ or $\mathfrak{so}(n,1)$, so we conclude by \cite[Theorem 4.0.1]{MR1852148} that $G$ has the Haagerup property, or equivalently $p'_G=0$. It remains to consider the case when $2<p'_{\mathfrak g}<\infty$. Let $\mathfrak{g}_H \subset \mathfrak{g}$ denote the sum of the $\mathfrak{g}_i$ with $p'_{\mathfrak{g_i}} = 0$, and $\mathfrak{g}_T$ the sum of the rest, that is the subalgebras isomorphic to $\mathfrak{sp}(n,1)$ or $\mathfrak{f}_4^{-20}$ (H for Haagerup and T for property (T)). Denote $G_H$ and $G_T$ the corresponding analytic subgroups of $G$. As explained above, we know from \cite{MR510552} that $Z(G_T)$ is finite, so by Proposition \ref{prop:stability_quotient_by_compact} we can assume that $Z(G_T) = \{1\}$. This means that $G$ is isomorphic to $G_H \times \mathrm{Ad}(\mathfrak{g}_T)$, so $p'_G = \max(p'_{G_H},p'_{\mathfrak{g}_T})$. This is equal to $p'_{\mathfrak g}$ as we already justified that $p'_{G_H}=0$.
    \end{proof}

We need the following classical fact. It is for example stated for higher-rank groups in \cite[Lemma 5.2]{dlSActa}, but the proof remains valid without any rank assumption.
  \begin{lemma}\label{lem:image_of_lattice_is_lattice} Let $G$ be a connected semisimple Lie group, and denote $q \colon G \to G/Z(G)$ the quotient map. If $\Gamma < G$ is a lattice, then $q(\Gamma) < G/Z(G)$ is a lattice and $\Gamma$ has finite index in $\Gamma Z(G)$.
  \end{lemma}

  \begin{thm}\label{thm:integrability} Let $G$ be a connected semisimple Lie group and $\Gamma < G$ be a lattice. Then $\Gamma$ is $L_p$-integrable for all $p < p_{\mathfrak g}$. Moreover, if $q(\Gamma)<G/Z(G)$ is an irreducible lattice, then $\Gamma$ is $L_p$-integrable for all $p < p'_{\mathfrak g}$.
  \end{thm}
  \begin{proof} It is enough to prove the Theorem when $G$ has trivial center. Indeed, in the notation of Lemma \ref{lem:image_of_lattice_is_lattice}, by the argument given in \cite[Proposition 7.1]{strongTsp4}, a lattice $\Gamma\subset G$ is $L_p$-integrable if the lattice $q(\Gamma) < G/Z(G)$ is $L_p$-integrable. The essential ingredient is the fact that the central extension $G \to Z(G)$ is given by a bounded $2$-cocycle. For simple Lie groups, one can justify this by refering to \cite{MR510552} (see also \cite{MR1857863}). The general case of semisimple Lie groups follows by decomposing the universal cover into simple parts.

    So assume that $G$ has trivial center. We can assume that $G$ has no compact factor. By decomposing into irreducible factors, we can assume that $\Gamma$ is irreducible. When $G$ has higher rank, we have more than the Theorem~: $\Gamma$ is $L_p$-integrable for every $p<\infty$. This is what the proof of the $L_2$-integrability in \cite[\S 2]{MR1767270} actually shows. See also the proof of \cite[Lemma 5.6]{dlSActa} for a simpler argument, showing that the $L_p$-integrability holds for every quasi-isometrically embedded lattice in a Lie group.  When $G$ has rank $1$, the only cases to consider are $\mathfrak{g} = \mathfrak{sp}(n,1)$, $n\geq 2$ and $\mathfrak{f}_4^{-20}$ as the other rank one simple Lie algebras satisfy $p_{\mathfrak{g}}=p'_{\mathfrak{g}}=0$. By Theorem \ref{L_p integrable rank one}, we know that every lattice in $\mathrm{Sp}(n,1)$ or $F_4^{-20}$ is $L_p$-integrable for all $p < 4n+2$ and $p < 22$ respectively. Since we also know that $p_{\mathrm{Sp}(n,1)} \leq 4n+2$ and $p_{F_4^{-20}} \leq 22$, thanks to Theorem \ref{pansu rank one}, this ends the proof.
  \end{proof}

    \begin{thm} Let $G$ be a connected semisimple Lie group with Lie algebra $\mathfrak{g}$ and let $\Gamma < G$ be a lattice. Then
    \begin{equation}\label{eq:value-lattice} p_\Gamma = p_G = p_{\mathfrak{g}} \quad  \text{and} \quad p'_\Gamma = p'_G = p'_{\mathfrak{g}}.\end{equation}
    \end{thm} 

  \begin{proof} The Theorem is clear if $G$ is compact, so we can and will assume that $G$ is not compact. We can and will even assume (replacing $G$ by its quotient by its compact factors) that $G$ does not have compact factors. Denote by $\widetilde G$ the universal cover of $G$, and $\widetilde q \colon \widetilde G \to G$ and $q \colon G \to G/Z(G) \simeq \mathrm{Ad}(G)$ the quotient maps. Let $\mathfrak{g} = \oplus_{i=1}^n \mathfrak{g_i}$ be the decomposition of the Lie algebra of $G$ as a direct sum of simple Lie algebras.
  
  Let $\Lambda:=q(\Gamma) < G/Z(G)$. When $\Lambda$ is an irreducible lattice, \eqref{eq:value-lattice} follows from Theorem \ref{thm:integrability}. In general, we reduce to the irreducible case. By Lemma \ref{lem:image_of_lattice_is_lattice}, $\Gamma$ is a lattice in $G/Z(G)$, and $\Gamma$ has finite index in $\Gamma Z(G) = q^{-1}(\Lambda)$. So replacing $\Gamma$ by $\Gamma Z(G)$, we can assume that $Z(G) \subset \Gamma$ and define $\widetilde \Gamma = \widetilde q^{-1}(\Gamma) = (q \circ \widetilde q)^{-1}(\Lambda)<\widetilde G$. Decompose $\Lambda$ into irreducible components. To do so, consider a maximal partition $(A_1,\dots,A_m)$ with the property that $\prod_k (\Lambda_k:= \Lambda \cap \mathrm{Ad}(\oplus_{i \in A_k} \mathfrak{g}_i))$ is of finite index in $\Lambda$. Then every $\Lambda_k$ is an irreducible lattice in $\mathrm{Ad}(\oplus_{i \in A_k} \mathfrak{g}_i)$, and replacing $\Gamma$ be a finite index subgroup, we can assume that $\Lambda = \prod_k \Lambda_k$.

The lattice $\Gamma$ is not necessarily a product, but this is true in the universal cover: if $\widetilde \Lambda_k$ denotes the lift of $\Lambda_k$ in the universal cover $\widetilde{G}_{A_k}$ of $\mathrm{Ad}(\oplus_{i \in A_k} \mathfrak{g}_i)$, then $\widetilde \Gamma = \prod_k \widetilde \Lambda_k$. It follows from Proposition \ref{prop:stability_product} and from the case of irreducible lattices that $p_{\Lambda} = \min_k p_{\Lambda_k} = p_{\mathfrak g}$ and that $p_{\widetilde \Gamma} = \min_k p_{\widetilde{\Lambda}_k} = p_{\mathfrak g}$, so we deduce $p_\Gamma = p_{\mathfrak g}$ by Proposition \ref{prop:subgroups_and_quotients}. For the invariant $p'_\Gamma$, the inequality $p'_\Gamma \leq p'_G = p'_{\mathfrak{g}}$ is by Proposition \ref{prop:subgroups_and_quotients}. For the converse, observe that the discrete subgroup $\widetilde{q}(\widetilde \Lambda_k)<\Gamma$ is a lattice in the quotient of $\widetilde{G}_{A_k}$ by $ \widetilde{G}_{A_k} \cap \ker q $, so we obtain from the case of irreducible lattices and Proposition \ref{prop:subgroups_and_quotients} that $\max_{i \in A_k} p_{\mathfrak{g}_i} = p_{q(\widetilde \Lambda_k)} \leq p'_\Gamma$.
    Taking the maximum over $k$ we obtain $p'_{\mathfrak g}  \leq p'_\Gamma$. This concludes the proof of  \eqref{eq:value-lattice} and of the Theorem.
\end{proof}

\appendix

\section{Rank one symmetric spaces and $L_p$-integrability.}
By \'Elie Cartan's classification (see Example \ref{example lie algebras} for references), every connected simple Lie group of real rank one is locally isomorphic to the isometry group of one of the following symmetric spaces of rank one :
\begin{itemize}
\item The real hyperbolic space $\R \mathbb{H}^n$ for $n \geq 2$, its isometry group is $\mathrm{SO}(n,1)$.
\item The complex hyperbolic space $\C \mathbb{H}^n$ for $n \geq 2$, its isometry group is $\mathrm{SU}(n,1)$.
\item The quaternionic hyperbolic space $\bH \mathbb{H}^n$ for $n \geq 2$, its isometry group is $\mathrm{Sp}(n,1)$.
\item The octonionic hyperbolic plane $\bO \mathbb{H}^2$, its isometry group is the exceptional Lie group $F_4^{-20}$.
\end{itemize}
Let $\K=\R,\C,\bH$ and $n \geq 2$ or $\K=\bO$ and $n=2$. Let $d=\dim_\R \K \in \{1,2,4,8\}$. We give a description of the hyperbolic space $\K \mathbb{H}^n$ that will be convenient for our study. Let $\langle \cdot, \cdot \rangle : \K^{n} \times \K^{n} \rightarrow \K$ be the standard sesquilinear form given by 
$$ \langle x, y \rangle = \overline{x_1}y_1+ \overline{x_1}y_1 + \overline{x_2}y_2 + \cdots + \overline{x_{n}}y_{n}.
$$
Define a nilpotent Lie group $N$ formally given by $N=\K^{n-1} \times \Im(\K)$ but with the group law
$$ (\xi_1,r_1) \cdot (\xi_2,r_2)=(\xi_1+\xi_2,r_1+r_2 + \Im \langle \xi_2,\xi_1 \rangle)$$
where $\langle \cdot, \cdot \rangle : \K^{n-1} \times \K^{n-1} \rightarrow \K$ is the standard sesquilinear form and $\Im$ denotes the imaginary part. We define an action of $\R$ by group automorphisms on $N$ with the formula
$$t \cdot (\xi,r)=(e^t \xi, e^{2t}r), \quad t \in \R, \; (\xi,r) \in N.$$
We let $P=N \rtimes \R$ be the corresponding semi-direct product.

\begin{thm}[Heintze \cite{heintze}]
The hyperbolic space $\K\mathbb{H}^n$ is isometric to the group $P$ equipped with the unique left invariant riemannian metric on $P$ that is given at the origin $(0,0,0) \in P$ by
$$ \rd s^2=\frac{1}{2}\| \mathrm{d} \xi \|^2+\frac{1}{4}| \mathrm{d} r|^2+\mathrm{d} t^2.$$
\end{thm}
Thanks to the description of $X=\K \mathbb{H}^n$ as a group with a left invariant metric, we can directly apply \cite[Proposition 24]{MR1086210} to get the following result.

\begin{thm}[Pansu \cite{MR1086210}] \label{pansu rank one}
Let $G$ be the isometry group of $X$. Then the cohomology of $G$ with values in $L_p(G)$ (for the left regular representation) is nonzero for all $p > d(n+1)-2$ where $d=\dim_\R \K$.
\end{thm}
\begin{proof}
We have $P=N \rtimes \R$ and the action of $\R$ on $N$ is generated by a derivation on $N=\K^{n-1} \times \Im(\K)$ that has eigenvalue $1$ on $\K^{n-1}$ and eigenvalue $2$ on $\Im(\K)$. It follows from \cite[Proposition 24]{MR1086210} that the $L_p$ cohomology of $X$ is nonzero for all 
$$p > \dim_\R \K^{n-1} +2 \dim_\R \Im(\K)= d(n-1)+2(d-1)=d(n+1)-2.$$
Finally, the $L_p$-cohomology of $G$ with values in the left regular representation can be identified with the $L_p$-cohomology of $X$, see \cite[Section 4]{MR2421319}.
\end{proof}

Let $X=\K\mathbb{H}^n$. Let $\partial X$ denote its boundary.
Fix $\omega \in \partial X$. We define the geodesic flow $(\pi_t)_{t \in \R}$ with source $\omega$ as follows : for every $x \in X$, the map $t \mapsto \pi_t(x)$ is the unit speed geodesic with $\pi_0(x)=x$ and $\lim_{t \to -\infty} \pi_t(x)=\omega$. We define an equivalence relation on $X$ by saying that $x,y \in X$ are equivalent if
$$ \lim_{t \to -\infty} d(\pi_t(x),\pi_t(y))=0.$$
The equivalence classes of this equivalence relation are called \emph{horospheres} centered at $\omega$.

It is easy to see from the definition of $P$ and the left invariance of its metric that $t \mapsto (0,0,t) \in P$ is a unit speed geodesic path. Fix an isometry $\tau : P \rightarrow X$ such that the geodesic line $t \mapsto \tau(0,0,t)$ converges to $\omega$ (this is possible because the isometry group of $X$ acts transitively on $\partial X$). We observe the following facts which can be deduced easily from the left invariance of the metric of $P$:
\begin{enumerate}
\item  We have $\pi_t(\tau(\xi,r,s))=\tau(\xi,r,s-t)$ for all $(\xi,r,s) \in P$.
\item For every $s \in \R$, the set $H_s=\{ \tau(\xi,r,s) \mid (\xi,r) \in N\}$ is a horosphere centered at $\omega$.
\end{enumerate}
Of course the map $\pi_t$ restricts to a bijection from $H_s$ onto $H_{s-t}$ for all $s,t \in \R$ but it is not an isometry. Instead we have the following property.

\begin{prop} \label{dilation}
Fix $\omega \in \partial X$ and let $(\pi_t)_{t \in \R}$ be the associated geodesic flow. Take $x \in X$ and let $H$ be the horosphere centered at $\omega$ and containing $x$. Then there exists an orthogonal decomposition of the tangent space $$\mathbf{T}H_x=V_1 \oplus V_2$$
such that $\| (\rd\pi_t)_x(v)\|=e^{kt}\|v\|$ for all $v \in V_k, \: k=1,2$ and all $t \in \R$. The space $V_1$ has dimension $d(n-1)$ and $V_2$ has dimension $d-1$.
\end{prop}
\begin{proof}
We may choose the isometry $\tau : P \rightarrow X$ such that $x=\tau(0,0,0)$. Define $V_1$ to be the subspace of $\mathbf{T}H_x$ tangent to $\{ \tau(\xi,0,0) \mid \xi \in \K^{n-1} \}$ and $V_2$ to be the subspace tangent to $\{ \tau(0,r,0) \mid r \in \Im(\K) \}$. Then $V_1$ and $V_2$ are orthogonal by the definition of the metric on $P$. By the left invariance of the metric on $P$, the map $$\alpha_t : P \ni \tau(\xi,r,s) \mapsto \tau((0,0,t) \cdot (\xi,r,s))=\tau(e^t\xi,e^{2t}r,s+t) \in P$$
is an isometry. Observe that for all $(\xi,r,s) \in P$, we have
$$(\alpha_t \circ \pi_t)(\xi,r,s)=(e^t\xi,e^{2t}r,s).$$
This shows that $\rd( \alpha_t \circ \pi_t)_x(v)=e^{kt}v$ for all $v \in V_k, \: k=1,2$. Since $\alpha_t$ is an isometry, the conlusion follows.
\end{proof}

For every horosphere $H \subset X$, the distribution $x \mapsto V_1(x) \subset \mathbf{T}H_x$, defined by the proposition above, is called the \emph{Carnot distribution} of $H$. It has codimension $d-1$. A key fact is that any two points $x,y \in H$ can be joined by a smooth path in $H$ that remains tangent to the Carnot distribution at every point. The Carnot distance $d_C(x,y)$ is then defined as the infinimum of the lengths of all such paths. It follows from Proposition \ref{dilation} that $d_C(\pi_t(x),\pi_t(y))=e^td_C(x,y)$ for all $x,y \in H$.

\begin{lemma} \label{geodesic projection dilation}
Fix $\omega \in \partial X$ and let $(\pi_t)_{t \in \R}$ be the associated geodesic flow. Let $H \subset X$ be a horosphere and let $D \subset H$ be a bounded open subset.
\begin{enumerate}
\item There exists a constant $C > 0$ such that the diameter of $\pi_t(D)$ (in the path metric of $\pi_t(H)$) is bounded by $C\exp(t)$ for all $t \in \R$.
\item There exists a constant $C' > 0$ such that the volume of $\pi_t(D)$ (with respect to the volume form of $\pi_t(H)$) is equal to $C'\exp((d(n+1) - 2 )t)$ for all $t \in \R$.
\end{enumerate}
\end{lemma}
\begin{proof}
(1) For the Carnot distance, the diameter of $\pi_t(D)$ is equal to the diameter of $D$ times $\exp(t)$. But by definition, the path distance is always less than the Carnot distance. Therefore the diameter of $\pi_t(D)$ for the the path distance of $\pi_t(H)$ is bounded by $C\exp(t)$ for some $C > 0$.

(2) Fix $t \in \R$. Let $m_0$ be the volume form of $H$ and $m$ be the volume form of $\pi_t(H)$. Then by Proposition \ref{dilation}, we have $(\pi_t)_* m_0 = e^{td(n-1) + 2t (d-1)} m$. This imples $$m(\pi_t(D))=e^{(d(n+1)-2)t} m_0(D).$$
\end{proof}

The proof of the following theorem is due to Rich Schwartz. The proof can be found in the appendix of \cite{10.2307/2661380}. We briefly argue that the proof of Schwartz can be adapted to include the exceptional octonionic case and the case $p \neq 2$ in a unified way.

\begin{thm}[Schwartz] \label{L_p integrable rank one}
Let $\Gamma < G=\mathrm{Isom}(X)$ be any lattice. Then it is $L_p$-integrable for all $p < d(n+1)-2$.
\end{thm}
\begin{proof}
We follow the notations and the proof of \cite[Theorem 3.7]{10.2307/2661380} given in the appendix. We have to show that the function $f$ defined there is $L_p$-integrable on $F^0$ and as explained in Lemma 2.1, for this it is enough to show that $\sum_{m \in \N} g(m)^pv(m) < +\infty$. Recall that $v(m)$ is the volume of
$$B(m)=\bigcup_{m \leq t \leq m+1} \pi_{-t}(F_0).$$
By Lemma \ref{geodesic projection dilation}, we know that
$$ v(m) = C'\int_{m}^{m+1} \exp \left( -(d(n+1)-2)t \right) \: \rd t = A'\exp \left( -(d(n+1)-2)m \right)$$
for some constant $A' > 0$. Fix any $x \in H_0$ and let $\alpha_t$ be the isometry defined in Proposition \ref{dilation}, with $\alpha_t(x)=\pi_{-t}(x)$ for all $t \in \R$. Then we can find a bounded open ball $D \subset H_0$ around $x$ such that the unit tubular neighborhood $B_1(m)$ is contained in
$$ \bigcup_{m-1 \leq t \leq m+2} \alpha_t(D)$$
for all $m \in \N$. Thus $\pi(B_1(m))$ is contained 
$$ \bigcup_{m-1 \leq t \leq m+2} \pi_t(\alpha_t(D))=\bigcup_{m-1 \leq t \leq m+2} \alpha_t(\pi_t(D))$$
Thanks to Lemma \ref{dilation} and since $\alpha_{t}$ is an isometry, the diameter of $\alpha_t(\pi_t(D))$ is bounded by $C \exp(t)$ for some $C > 0$. Since $x \in \alpha_t(\pi_t(D))$ for all $t$, it follows that the diameter of $\pi(B_1(m))$ is bounded by $A \exp(m)$ for some constant $A > 0$. Finally, we conclude that
$$ \sum_{m \in \N} g(m)^pv(m) \leq A^pA' \exp(pm) \exp(-(d(n+1)-2)m) < +\infty$$
whenever $p < d(n+1)-2$.
\end{proof}

\bibliographystyle{plain_correctalpha} 
\bibliography{biblio}

\def\cprime{$'$}
\begin{thebibliography}{10}

\bibitem{arano2019ergodic}
Yuki Arano, Yusuke Isono, and Amine Marrakchi.
\newblock Ergodic theory of affine isometric actions on hilbert spaces, 2019.

\bibitem{MR2929085}
U.~Bader, T.~Gelander, and N.~Monod.
\newblock A fixed point theorem for {$L^1$} spaces.
\newblock {\em Invent. Math.}, 189(1):143--148, 2012.

\bibitem{MR2316269}
Uri Bader, Alex Furman, Tsachik Gelander, and Nicolas Monod.
\newblock Property ({T}) and rigidity for actions on {B}anach spaces.
\newblock {\em Acta Math.}, 198(1):57--105, 2007.

\bibitem{MR3117525}
Uri Bader, Alex Furman, and Roman Sauer.
\newblock Integrable measure equivalence and rigidity of hyperbolic lattices.
\newblock {\em Invent. Math.}, 194(2):313--379, 2013.

\bibitem{MR1979183}
Marc Bourdon and Herv\'{e} Pajot.
\newblock Cohomologie {$l_p$} et espaces de {B}esov.
\newblock {\em J. Reine Angew. Math.}, 558:85--108, 2003.

\bibitem{MR1509178}
Elie Cartan.
\newblock Les groupes r\'{e}els simples, finis et continus.
\newblock {\em Ann. Sci. \'{E}cole Norm. Sup. (3)}, 31:263--355, 1914.

\bibitem{MR2671183}
Indira Chatterji, Cornelia Dru\c{t}u, and Fr\'{e}d\'{e}ric Haglund.
\newblock Kazhdan and {H}aagerup properties from the median viewpoint.
\newblock {\em Adv. Math.}, 225(2):882--921, 2010.

\bibitem{MR1852148}
Pierre-Alain Cherix, Michael Cowling, Paul Jolissaint, Pierre Julg, and Alain
  Valette.
\newblock {\em Groups with the {H}aagerup property}, volume 197 of {\em
  Progress in Mathematics}.
\newblock Birkh\"{a}user Verlag, Basel, 2001.
\newblock Gromov's a-T-menability.

\bibitem{MR3590529}
Alan Czuro\'{n}.
\newblock Property {$F\ell_q$} implies property {$F\ell_p$} for
  {$1<p<q<\infty$}.
\newblock {\em Adv. Math.}, 307:715--726, 2017.

\bibitem{CzuronKalentar}
Alan Czuro\'{n} and Mehrdad Kalantar.
\newblock On fixed point property for $l_p$-representations of kazhdan groups.
\newblock {\em arXiv:2007.15168}, 2020.

\bibitem{MR2421319}
Yves De~Cornulier, Romain Tessera, and Alain Valette.
\newblock Isometric group actions on {B}anach spaces and representations
  vanishing at infinity.
\newblock {\em Transform. Groups}, 13(1):125--147, 2008.

\bibitem{MR3753580}
Cornelia Dru\c{t}u and Michael Kapovich.
\newblock {\em Geometric group theory}, volume~63 of {\em American Mathematical
  Society Colloquium Publications}.
\newblock American Mathematical Society, Providence, RI, 2018.
\newblock With an appendix by Bogdan Nica.

\bibitem{MR3872847}
Cornelia Dru\c{t}u and John~M. Mackay.
\newblock Random groups, random graphs and eigenvalues of {$p$}-{L}aplacians.
\newblock {\em Adv. Math.}, 341:188--254, 2019.

\bibitem{DruNo}
Cornelia Dru\c{t}u and Piotr~W. Nowak.
\newblock Kazhdan projections, random walks and ergodic theorems.
\newblock {\em J. Reine Angew. Math.}, 754:49--86, 2019.

\bibitem{MR2198325}
David Fisher and Gregory Margulis.
\newblock Almost isometric actions, property ({T}), and local rigidity.
\newblock {\em Invent. Math.}, 162(1):19--80, 2005.

\bibitem{MR1740985}
Alex Furman.
\newblock Orbit equivalence rigidity.
\newblock {\em Ann. of Math. (2)}, 150(3):1083--1108, 1999.

\bibitem{MR1253544}
M.~Gromov.
\newblock Asymptotic invariants of infinite groups.
\newblock In {\em Geometric group theory, {V}ol. 2 ({S}ussex, 1991)}, volume
  182 of {\em London Math. Soc. Lecture Note Ser.}, pages 1--295. Cambridge
  Univ. Press, Cambridge, 1993.

\bibitem{MR510552}
A.~Guichardet and D.~Wigner.
\newblock Sur la cohomologie r\'eelle des groupes de {L}ie simples r\'eels.
\newblock {\em Ann. Sci. \'Ecole Norm. Sup. (4)}, 11(2):277--292, 1978.

\bibitem{MR560633}
Uffe Haagerup.
\newblock {$L^{p}$}-spaces associated with an arbitrary von {N}eumann algebra.
\newblock In {\em Alg\`ebres d'op\'{e}rateurs et leurs applications en physique
  math\'{e}matique ({P}roc. {C}olloq., {M}arseille, 1977)}, volume 274 of {\em
  Colloq. Internat. CNRS}, pages 175--184. CNRS, Paris, 1979.

\bibitem{heintze}
Ernst Heintze.
\newblock On homogeneous manifolds of negative curvature.
\newblock {\em Mathematische Annalen}, 211(1):23--34, 1974.

\bibitem{MR1920389}
Anthony~W. Knapp.
\newblock {\em Lie groups beyond an introduction}, volume 140 of {\em Progress
  in Mathematics}.
\newblock Birkh\"{a}user Boston, Inc., Boston, MA, second edition, 2002.

\bibitem{strongTsp4}
Tim de~Laat and Mikael de~la Salle.
\newblock Strong property ({T}) for higher-rank simple {L}ie groups.
\newblock {\em Proc. Lond. Math. Soc. (3)}, 111(4):936--966, 2015.

\bibitem{L2spectralGap}
Tim de~Laat and Mikael de~la Salle.
\newblock Banach space actions and l2-spectral gap.
\newblock {\em Anal. PDE}, in press, 2020.

\bibitem{lavyOlivier}
Omer Lavy and Baptiste Olivier.
\newblock Fixed-point spectrum for group actions by affine isometries on
  {L}p-spaces.
\newblock {\em Ann. Inst. Fourier (Grenoble)}, to appear, 2020.

\bibitem{MR3987868}
Amine Marrakchi.
\newblock Spectral gap characterization of full type {III} factors.
\newblock {\em J. Reine Angew. Math.}, 753:193--210, 2019.

\bibitem{MR2353912}
Michael Megrelishvili.
\newblock Reflexively representable but not {H}ilbert representable compact
  flows and semitopological semigroups.
\newblock {\em Colloq. Math.}, 110(2):383--407, 2008.

\bibitem{MR2101227}
Manor Mendel and Assaf Naor.
\newblock Euclidean quotients of finite metric spaces.
\newblock {\em Adv. Math.}, 189(2):451--494, 2004.

\bibitem{MR2783928}
Assaf Naor and Yuval Peres.
\newblock {$L_p$} compression, traveling salesmen, and stable walks.
\newblock {\em Duke Math. J.}, 157(1):53--108, 2011.

\bibitem{MR3069362}
Bogdan Nica.
\newblock Proper isometric actions of hyperbolic groups on {$L^p$}-spaces.
\newblock {\em Compos. Math.}, 149(5):773--792, 2013.

\bibitem{MR2270593}
Piotr~W. Nowak.
\newblock Group actions on {B}anach spaces and a geometric characterization of
  a-{T}-menability.
\newblock {\em Topology Appl.}, 153(18):3409--3412, 2006.

\bibitem{MR3382026}
Piotr~W. Nowak.
\newblock Group actions on {B}anach spaces.
\newblock In {\em Handbook of group actions. {V}ol. {II}}, volume~32 of {\em
  Adv. Lect. Math. (ALM)}, pages 121--149. Int. Press, Somerville, MA, 2015.

\bibitem{MR2957217}
Baptiste Olivier.
\newblock Kazhdan's property {$(T)$} with respect to non-commutative
  {$L_p$}-spaces.
\newblock {\em Proc. Amer. Math. Soc.}, 140(12):4259--4269, 2012.

\bibitem{MR1086210}
Pierre Pansu.
\newblock Cohomologie {$L^p$} des vari\'{e}t\'{e}s \`a courbure n\'{e}gative,
  cas du degr\'{e} {$1$}.
\newblock Number Special Issue, pages 95--120 (1990). 1989.
\newblock Conference on Partial Differential Equations and Geometry (Torino,
  1988).

\bibitem{MR4014781}
Mikael de~la Salle.
\newblock A local characterization of {K}azhdan projections and applications.
\newblock {\em Comment. Math. Helv.}, 94(3):623--660, 2019.

\bibitem{dlSActa}
Mikael de~la Salle.
\newblock Strong property ({T}) for higher-rank lattices.
\newblock {\em Acta Math.}, 223(1):151--193, 2019.

\bibitem{MR1767270}
Yehuda Shalom.
\newblock Rigidity of commensurators and irreducible lattices.
\newblock {\em Invent. Math.}, 141(1):1--54, 2000.

\bibitem{10.2307/2661380}
Yehuda Shalom.
\newblock Rigidity, unitary representations of semisimple groups, and
  fundamental groups of manifolds with rank one transformation group.
\newblock {\em Annals of Mathematics}, 152(1):113--182, 2000.

\bibitem{MR1857863}
A.~I. Shtern.
\newblock Bounded continuous real 2-cocycles on simply connected simple {L}ie
  groups and their applications.
\newblock {\em Russ. J. Math. Phys.}, 8(1):122--133, 2001.

\bibitem{MR1943006}
M.~Takesaki.
\newblock {\em Theory of operator algebras. {II}}, volume 125 of {\em
  Encyclopaedia of Mathematical Sciences}.
\newblock Springer-Verlag, Berlin, 2003.
\newblock Operator Algebras and Non-commutative Geometry, 6.

\bibitem{MR248514}
L.~Tzafriri.
\newblock Remarks on contractive projections in {$L_{p}$}-spaces.
\newblock {\em Israel J. Math.}, 7:9--15, 1969.

\bibitem{MR550072}
William~A. Veech.
\newblock Weakly almost periodic functions on semisimple {L}ie groups.
\newblock {\em Monatsh. Math.}, 88(1):55--68, 1979.

\bibitem{MR611284}
F.~J. Yeadon.
\newblock Isometries of noncommutative {$L^{p}$}-spaces.
\newblock {\em Math. Proc. Cambridge Philos. Soc.}, 90(1):41--50, 1981.

\bibitem{MR2221161}
Guoliang Yu.
\newblock Hyperbolic groups admit proper affine isometric actions on
  {$l^p$}-spaces.
\newblock {\em Geom. Funct. Anal.}, 15(5):1144--1151, 2005.

\end{thebibliography}

\end{document}